\documentclass[11pt,a4paper]{article}
\usepackage{amsmath,amscd}
\usepackage{amsthm}
\usepackage{amssymb}		%Para escribir las letras griegas en negrita
\usepackage{enumitem}
\usepackage{bm}
\usepackage{mathalfa}
\usepackage{enumitem}
\usepackage{mathtools}
\usepackage{textcomp }
\DeclareMathAlphabet\mathbfcal{OMS}{cmsy}{b}{n}

\usepackage{color}

\theoremstyle{plain}
  \newtheorem{theorem}{\bf Theorem}[section]
  \newtheorem{proposition}[theorem]{\bf Proposition}
  \newtheorem{lemma}[theorem]{\bf Lemma}

\theoremstyle{remark}
  \newtheorem{remark}[theorem]{\bf Remark}
  \newtheorem{definition}[theorem]{\bf Definition}
\begin{document}
	\let\svthefootnote\thefootnote
%%%%%%%%%%%%%%%%%%%%%%%%%%%%%%%%%%%%%%%%%%%%%%%%%%%%%%%%%%%%%%%%%%%%%%%%%%%%%%%%%%%%%%%%%
\title{Reproducing kernel for elastic Herglotz functions.}

\author{T. Luque\thanks{Departamento de An\'alisis Matem\'atico y Matem\'atica Aplicada, Facultad de Matem\'aticas, Universidad Complutense de Madrid. E-mail: t.luque@ucm.es.} \and M. C. Vilela\thanks{Departamento de Matem\'atica e Inform\'atica aplicadas a las Ingenier\'ias Civil y Naval, Universidad Polit\'ecnica de Madrid. E-mais: maricruz.vilela@upm.es.}}

\maketitle
%%%%%%%%%%%%%%%%%%%%%%%%%%%%%%%%%%%%%%%%%%%%%%%%%%%%%%%%%%%%%%%%%%%%%%%%%%%%%%%%%%%%%%%%%%%%

\begin{abstract}
We study the elastic Herglotz wave functions, which are entire solutions of the 
spectral Navier equation appearing in the linearized elasticity theory with $L^2-$far-field patterns. 
We characterize in three-dimensions the set of these functions $\mathcal{W},$ as a close subspace of a Hilbert space $\mathcal{H}$ of vector valued functions such that they and their spherical gradients belong to a certain weighted $L^2$ space. This allows us to prove that $\mathcal{W}$ is a reproducing kernel Hilbert space and to calculate the reproducing kernel. Finally, we outline the proof for the two-dimensional case and give the corresponding reproducing kernel. 
\end{abstract}
%%%%%%%%%%%%%%%%%%%%%%%%%%%%%%%%%%%%%%%%%%%%%%%%%%%%%%%%%%%%%%%%%%%%%%%%%%%%%%%%%%%
\let\thefootnote\relax\footnote{{\bf Funding}:
	The first author was supported by Spanish Grant MTM2017-82160-C2-1-P, and
	the second by Spanish Grant MTM2014-57769-C3-1-P and MTM2017-85934-C3-3-P.
}
\addtocounter{footnote}{-1}\let\thefootnote\svthefootnote
\section{Introduction and statement of results}\label{intro}
The Navier equation of the dynamic linearized elasticity in a three-dimensional homogeneous and
isotropic medium is given by
\begin{equation}
\mu\Delta_{\bm{x}}\mathbf{U}(\bm{x},t)+(\lambda+\mu)\nabla_{\bm{x}}\left( \nabla_{\bm{x}}\cdot\mathbf{U}%
(\bm{x},t)\right) =\rho\ \frac{\partial^{2}}{\partial t^{2}}\mathbf{U}(\bm{x},t),
\label{Navier}
\end{equation}
where $\mathbf{U}(\bm{x},t):\mathbb{R}^3\times\mathbb{R}\longrightarrow\mathbb{C}^3$ is the time-dependent displacement field, $\rho>0$ is the constant mass
density of the medium, and $\lambda,\mu$ are the so called Lam\'{e} constants.
Here $\Delta_{\bm{x}}\mathbf{U}(\bm{x},t)$ denotes the vector Laplacian of $\mathbf{U}%
(\bm{x},t)$ with respect to the spatial variable $\bm{x},$ that is, the vector whose
components are the Laplacian with respect to the variable $\bm{x}$ of the
components of the vector $\mathbf{U}(\bm{x},t).$
0
For time-harmonic solutions of equation (\ref{Navier}), that is, for
solutions of the form $\mathbf{U}(\bm{x},t)=\mathbf{u}(\bm{x})e^{-i\omega t},$ with $%
\omega >0,$ we obtain the \emph{homogeneous spectral Navier equation}%
\begin{equation} 
\Delta ^{\ast }\mathbf{u}(\bm{x})+\rho \ \omega ^{2}\mathbf{\ u}(\bm{x})=0,
\label{spectralNavier}
\end{equation}
where%
\begin{equation*}
\Delta ^{\ast }\mathbf{u}(\bm{x})=\mu \Delta \mathbf{u}(\bm{x})+(\lambda +\mu )\nabla
(\nabla\cdot\mathbf{\,u})(\bm{x}).
\end{equation*}
Throughout this paper we will assume that the Lam\'e constants satisfy $\mu >0$ and $2\mu +\lambda >0,$ so that $\Delta^*$ is a strongly elliptic operator and, we will denote by $k_p$ and $k_s$ the speed of propagation of the longitudinal and transverse waves  respectively, which are given by
\begin{equation*}
k_p^2=\frac{\omega^2\rho}{2\mu+\lambda},
\qquad
k_s^2=\frac{\omega^2\rho}{\mu}.
\end{equation*}
We remark that in longitudinal waves, the oscillations occur in the longitudinal direction or the direction of the wave propagation. Since the compressional forces are active on these waves, they are also called compressional waves. If instead  we look for a transverse or share wave, the particles oscillate at a right angle or perpendicular to the direction of propagation. If we use the Helmholtz decomposition, we can write any solution $\mathbf{u}$ of \eqref{spectralNavier} as the sum 
\begin{equation}
\label{descomposicionHelmholtz}
\mathbf{u}=\mathbf{u}_p+\mathbf{u}_s,
\end{equation}
where $\mathbf{u}_p$ and $\mathbf{u}_s$ are called the compressional part and  the shear part of $\bf u$ respectively and are given by
\begin{equation}
\label{upus}
\mathbf{u}_p=\frac{1}{k_p^2}\nabla(\nabla\cdot\mathbf{u}),\qquad \mathbf{u}_s=\frac{1}{k_s^2}\nabla\times(\nabla\times\mathbf{u}).
\end{equation}
Observe that $\mathbf{u}_p$ and $\mathbf{u}_s$ trivially satisfy
$$
\nabla\times\mathbf{u}_p=0,\qquad\nabla\cdot\mathbf{u}_s=0;
$$
as well as they are solutions of the following \emph{vectorial homogeneous Helmholtz equations}
\begin{eqnarray}
\label{eqcompressional}
&&\Delta \mathbf{u}_p+k_p^2\mathbf{u}_p=0,
\\
\label{eqshear}
&&\Delta \mathbf{u}_s+k_s^2\mathbf{u}_s=0.
\end{eqnarray}
Moreover, the wave potential $f=\nabla\cdot\mathbf{u}$ is a solution of the scalar homogeneous Helmholtz equation with wave number $k_p,$ and $\nabla\times\mathbf{u}$ is a solution of the vectorial homogeneous Helmholtz equation with wave number $k_s.$

In order to define the three dimensional elastic Herglotz wave functions we introduce the spaces $\mathbf{L}^2(\mathbb{S}^2):=[L^2(\mathbb{S}^2)]^3$ with $\mathbb{S}^2=\{\boldsymbol{\xi}\in\mathbb{R}^3\,:\, |\boldsymbol{\xi}|=1\},$ and
\begin{eqnarray*}
	&&\mathbf{L}_p^2(\mathbb{S}^2):=\{\mathbf{g}\in\mathbf{L}^2(\mathbb{S}^2)\,:\,\boldsymbol{\xi}\times\mathbf{g}(\boldsymbol{\xi})=\mathbf{0}\},
	\\
	&&\mathbf{L}_s^2(\mathbb{S}^2):=\{\mathbf{g}\in\mathbf{L}^2(\mathbb{S}^2)\,:\,\boldsymbol{\xi}\cdot\mathbf{g}(\boldsymbol{\xi})=0\}.
\end{eqnarray*}
Observe that $\mathbf{L}^2(\mathbb{S}^2)=\mathbf{L}_p^2(\mathbb{S}^2)\oplus\mathbf{L}_s^2(\mathbb{S}^2),$ where $\oplus$ denotes a direct sum.

\begin{definition}
	\label{ElasticHerglotz3d}
	We define a three-dimensional elastic Herglotz wave function as an entire (defined on the whole of $\mathbb{R}^3$) solution $\mathbf{u}$ of the spectral Navier equation \eqref{spectralNavier} such that 
	\begin{equation}
	\label{elasticHerglotz}
	\mathbf{u}(\bm{x})
	=
	\int_{\mathbb{S}^2}
	\left(
	e^{ik_p\bm{x}\cdot\boldsymbol{\xi}}\mathbf{g}_1(\boldsymbol{\xi})
	+
	e^{ik_s\bm{x}\cdot\boldsymbol{\xi}}\mathbf{g}_2(\boldsymbol{\xi})
	\right)
	\,d\sigma(\boldsymbol{\xi}),
	\end{equation}
	where $d\sigma$ is the associated measure of $\mathbb{S}^2,$ 
	$\mathbf{g}_1\in \mathbf{L}_p^2(\mathbb{S}^2)$ is the longitudinal far-field pattern of $\mathbf{u},$ and
	$\mathbf{g}_2\in \mathbf{L}_s^2(\mathbb{S}^2)$ is the transverse one.
\end{definition}

From this definition, we obtain by the above remarks that every elastic Herglotz wave function can be represented as
\begin{equation*}
\label{elasticHerglotz2}
\mathbf{u}(\bm{x})
=
\int_{\mathbb{S}^2}
\left(
e^{ik_p\bm{x}\cdot\boldsymbol{\xi}}g_p(\boldsymbol{\xi})\,\boldsymbol{\xi}
+
e^{ik_s\bm{x}\cdot\boldsymbol{\xi}}\,\boldsymbol{\xi}\times\mathbf{g}_s(\boldsymbol{\xi})
\right)
\,d\sigma(\boldsymbol{\xi}),
\end{equation*}
where $g_p\in L^2(\mathbb{S}^2)$ is the far-field pattern of the wave potential $f=\nabla\cdot\mathbf{u},$ and $\mathbf{g}_s\in \mathbf{L}^2(\mathbb{S}^2)$ is the far-field pattern of $\nabla\times\mathbf{u}.$

In the early sixties, Hartman and Wilcox \cite{HartmanWilcox} introduced for the first time the concept of Herglotz wave function. They were defined as the entire solutions $u$ of the homogeneous Helmholtz equation $\Delta u+k^2u=0$, $k>0$, such that 
\begin{equation}
\label{HerglotzConditionHW}
\|u\|_A^2
:=
\limsup\limits_{R\rightarrow\infty}\,\frac{1}{R}
\int_{|x|<R}|u(x)|^2\,dx
<\infty.
\end{equation}
Moreover, in \cite[Theorem 4]{HartmanWilcox}, they characterize the Herglotz wave functions as the entire solutions $u$ of the homogeneous Helmholtz equation with far-field pattern in $L^2(\mathbb{S}^{n-1})$; that is, the functions $u$ that can be written as 
\begin{equation}\label{HW:farfield}
u(x)=\int_{\mathbb{S}^{n-1}} e^{ikx\cdot\xi}g(\xi)\,d\sigma(\xi),
\end{equation}
where $g\in L^2(\mathbb{S}^{n-1})$. With this definition, they extended and completed previous studies by Herglotz in 1945, about asymptotic properties and boundedness conditions for solutions of the Helmholtz equation, and by M\"uller  \cite{Muller} in 1952, about properties of entire solutions of the wave equation. 

Since then, different characterizations of these functions have been given and they have been studied in many different contexts. For example, in the language of Harmonic Analysis, Herglotz wave functions are related to the Fourier extension operator and this connection has been explored in \cite{A+F-G+P-E}, \cite{BBR} and \cite{P-E+V-D}. Specifically, in this last paper, the authors characterize the Herglotz wave functions as the entire solutions $u$ of the homogeneous Helmholtz equation such that 
\begin{equation}
\label{salvador}
\|u\|_{\mathcal{H}(\mathbb{R}^n)}^2
:=\int_{|x|>1}\left(|u(x)|^2+\left|\frac{\partial u}{\partial r}(x)\right|+\left|\frac{\partial u}{\partial \theta}(x)\right|\right)\,dx
<\infty,
\end{equation}
where $(r,\theta)$ are the polar coordinates in $\mathbb{R}^n$, $n>2.$ The proof can be found in \cite[Theorem 4]{P-E+V-D} and the case $n=2$ was first proved in \cite{A+F-G+P-E}. 

From the point of view of applications, Herglotz wave functions play an important role in the study of wave scattering problems, specially in the development of the multiple scattering theory of acoustic and electromagnetic waves.  They have also become a necessary ingredient in certain reconstruction methods for inverse scattering problems known as linear sampling methods. For more details on this subject see  \cite{Arens} and also the monograph \cite[pp.56-59 and pp.217-220]{Colton}.

In the 1990s, Dassios and Rigou \cite{DassiosRigou} extended the Herglotz wave function concept to the context of elasticity.  They defined the three-dimensional elastic Herglotz wave functions as entire solutions $\mathbf{u}$ of the spectral Navier equation satisfying the so called Herglotz condition given by
\begin{equation}
\label{HerglotzCondition}
\|\mathbf{u}\|_A^2
:=
\limsup\limits_{R\rightarrow\infty}\,\frac{1}{R}
\int_{|\bm{x}|<R}|\mathbf{u}(\bm{x})|^2\,d\bm{x}
<\infty.
\end{equation}
In \cite[Theorem4]{DassiosRigou}, they proved a one-to-one correspondence between elastic Herglotz wave functions and their far-field patterns. Indeed, they prove that this definition of elastic herglotz wave function is equivalent to Definition \ref{ElasticHerglotz3d} and
\begin{equation}\label{upusfinita}
\|{\bf u}\|_A^2\sim \|\mathbf{g}_1\|_{L^2(\mathbb{S}^2)}^2+\|\mathbf{g}_2\|_{L^2(\mathbb{S}^2)}^2\sim \|{\bf u_p}\|_A^2+\|{\bf u_s}\|_A^2.
\end{equation}
The two-dimensional case was treated later by Sevroglou and Pelekanos in \cite{Pelekanos}, extending the results by Dassios and Rigou to dimension two. 

Later on, in \cite{2d}, Barcel\'o et al. characterized the two-dimensional elastic Herglotz wave functions $\mathbf{u},$ 
in terms of a weighted $L^2$ norm involving $\mathbf{u}$ and its angular derivative $\partial_{\varphi}\mathbf{u}$
(see Theorem 1.1 of \cite{2d}\footnote{We have found a misprint on identity (6) of \cite{2d}. There, the factor $|x|$ have to be removed.}).
More precisely, given 
$\mathbf{u}=(u^1,u^2),$ they defined its angular derivative as
\begin{equation}
\label{DerivadaAngular}
\partial_{\varphi}\mathbf{u}(\bm{x})
:=(\bm{x}^{\perp}\cdot\nabla u^1(\bm{x}),\bm{x}^{\perp}\cdot\nabla u^2(\bm{x}))
\end{equation}
with
$\bm{x}^{\perp}$
denoting the vector $\bm{x}$ rotated anti-clockwise by $\pi/2.$
With this notation, they characterized the two-dimensional elastic Herglotz wave functions 
as entire solutions $\mathbf{u},$  of the two-dimensional spectral Navier equation such that
\begin{equation}
\label{norma2d}
\| \mathbf{u}\| _{\mathcal{H}(\mathbb{R}^2)}^2
:=
\int_{\mathbb{R}^{2}}
\left( |\mathbf{u}(\bm{x})|^2+|\partial_{\varphi}\mathbf{u}(\bm{x})|^2\right) 
\frac{d\bm{x}}{\langle\bm{x}\rangle^3}
<\infty,
\end{equation}
where 
$
\langle\bm{x}\rangle:=(1+|\bm{x}|^2)^{1/2}.
$

In this paper, we extend this result to the three-dimensional case. 
In order to do that, it is convenient to rewrite \eqref{DerivadaAngular} using the polar coordinates $(r,\varphi),$ with $r$ and $\varphi$ denoting the radial and angular coordinates respectively. 

With this notation, the gradient operator can be written as
\begin{equation}\label{gradiente}
\nabla 
=\widehat{\mathbf r}\,\frac{\partial}{\partial r}
+\frac{\widehat{\bm{\varphi}}}{r}\,\frac{\partial}{\partial\varphi},
\end{equation}

where $\widehat{\mathbf r}$ and $\widehat{\bm{\varphi} }$ are the unit vectors in the radial and angular direction respectively, which are given by
$$
\widehat{\mathbf r}=\frac{\bm{x}}{|\bm{x}|}=(\cos\varphi,\sin\varphi)
\qquad
\text{and}
\qquad
\widehat{\bm{\varphi} }=\frac{\bm{x}^{\perp}}{|\bm{x}|}=(-\sin\varphi,\cos\varphi),
$$
and therefore
$$
\partial_{\varphi}\mathbf{u}
=\left(\frac{\partial  u^1}{\partial \varphi},\frac{\partial u^2}{\partial \varphi} \right).
$$
Writing $\mathbf{u}=u^r\,\widehat{\mathbf r}+u^\varphi\,\widehat{\bm{\varphi}},$
we have that
$$
u^1=\cos\varphi\,u^r-\sin\varphi\,u^\varphi
\qquad\text{and}\qquad
u^2=\sin\varphi\,u^r+\cos\varphi\,u^\varphi,
$$
and thus
\begin{equation*}
\label{NormNabla2d}
|\partial_{\varphi}\mathbf{u}|^2=
\left|\frac{\partial u^\varphi}{\partial\varphi}+u^r\right|^2+
\left|\frac{\partial u^r}{\partial\varphi}-u^\varphi\right|^2.
\end{equation*}
Now we introduce the second order tensor $\nabla_{\varphi},$ defined by
\begin{equation}
\label{NablaAngular2d}
\nabla_{\varphi}\mathbf{u}
=
\widehat{\bm{\varphi}}\,\frac{\partial}{\partial\varphi}\otimes\mathbf{u}
:=
\left(\frac{\partial u^r}{\partial\varphi}-u^\varphi\right)\,\widehat{\bm{\varphi}}\otimes\widehat{\mathbf{r}}
+
\left(\frac{\partial u^\varphi}{\partial\varphi}+u^r\right)\,\widehat{\bm{\varphi}}\otimes\widehat{\bm{\varphi}}.
\end{equation}
With this notation, we have that
\begin{equation}
\label{LoMismo}
|\partial_{\varphi}\mathbf{u}|=|\nabla_{\varphi}\mathbf{u}|.
\end{equation}
Here, given two second-order tensors 
$\mathrm{A}=A_{ij}\,\mathbf{e}_i\otimes\mathbf{e}_j,$
$\mathrm{B}=B_{ij}\,\mathbf{e}_i\otimes\mathbf{e}_j,$
where we are using Einstein notation and $\{\mathbf{e}_1, \mathbf{e}_2,\mathbf{e}_3\}$ is an orthonormal basis,
we introduce the double inner product of these two tensors, given by 
$
\mathrm{A}:\mathrm{B}=A_{ij}B_{ij},
$
and the norm of $\mathrm{A}$ which is given by
\begin{equation}
\label{NormaTensor}
|\mathrm{A}|=\sqrt{\mathrm{A}:\overline{\mathrm{A}}}.
\end{equation}

In the three-dimensional case we use spherical coordinates $(r,\theta,\varphi),$ such that
\begin{equation}
\label{CoordenadasEsfericas}
\bm{x}=(r\sin\theta\cos\varphi,r\sin\theta\sin\varphi,r\cos\theta),\quad r> 0, 0<\theta<\pi,0<\varphi<2\pi.
\end{equation}
With this notation, the gradient operator can be written as
$$\nabla 
=\widehat{\mathbf r}\,\frac{\partial}{\partial r}
+\frac{1}{r}\,
\left(
\widehat{\bm{\theta}}\,\frac{\partial}{\partial\theta}
+
\frac{\widehat{\bm{\varphi}}}{\sin\theta}\,\frac{\partial}{\partial\varphi}
\right),$$
where $\widehat{\mathbf r},$ $\widehat{\bm{\theta} }$ and $\widehat{\bm{\varphi} }$ are the unit vectors in the radial and angular directions respectively.

Now we introduce the second-order tensor $\nabla_S,$ similar to the one given in \eqref{NablaAngular2d}, defined by
\begin{eqnarray}
\label{NablaAngular3d}
\nabla_S\mathbf{u}
&=&
\left(
\widehat{\bm{\theta}}\,\frac{\partial}{\partial\theta}
+
\frac{\widehat{\bm{\varphi}}}{\sin\theta}\,\frac{\partial}{\partial\varphi}
\right)\otimes\mathbf{u}
\nonumber
\\
&:=&
\left(\frac{\partial u^r}{\partial\theta}-u^\theta\right)\,\widehat{\bm{\theta}}\otimes\widehat{\mathbf{r}}
+
\left(\frac{\partial u^\theta}{\partial\theta}+u^r\right)\,\widehat{\bm{\theta}}\otimes\widehat{\bm{\theta}}
+
\frac{\partial u^\varphi}{\partial\theta}\,\widehat{\bm{\theta}}\otimes\widehat{\bm{\varphi}}
\nonumber
\\
&&+
\left(\frac{1}{\sin\theta}\,\frac{\partial u^r}{\partial\varphi}-u^\varphi\right)\,\widehat{\bm{\varphi}}\otimes\widehat{\mathbf{r}}
+
\left(\frac{1}{\sin\theta}\,\frac{\partial u^\theta}{\partial\varphi}-u^\varphi\cot\theta\right)\,\widehat{\bm{\varphi}}\otimes\widehat{\bm{\theta}}
\nonumber
\\
&&+
\left(\frac{1}{\sin\theta}\,\frac{\partial u^\varphi}{\partial\varphi}+u^r+u^\theta\cot\theta\right)\,\widehat{\bm{\varphi}}\otimes\widehat{\bm{\varphi}}.
\end{eqnarray}
We will refer to $\nabla_S\mathbf{u}$ as the spherical gradient of $\mathbf{u}.$

We introduce the main result of this paper.
\begin{theorem}
	\label{caracterizacion}
	An entire solution $\mathbf{u}$ of the spectral Navier equation is a three-di\-men\-sio\-nal elastic Herglotz wave function if and only if
	\begin{equation}
	\label{norma3d}
	\| \mathbf{u}\| _{\mathcal{H}(\mathbb{R}^3)}^2
	:=
	\int_{\mathbb{R}^{3}}
	\left( |\mathbf{u}(\bm{x})|^2+\left|\nabla_S\mathbf{u}(\bm{x})\right|^2\right) 
	\frac{d\bm{x}}{\langle\bm{x}\rangle^3}
	<\infty,
	\end{equation}	
	where $\langle\bm{x}\rangle:=(1+|\bm{x}|^2)^{1/2}$, $\nabla_S\mathbf{u}$ is given by \eqref{NablaAngular3d} and its norm by \eqref{NormaTensor}.
	
	Moreover,
	\begin{equation*}
	%\label{equivalencias}
	\| \mathbf{u}\| _{\mathcal{H}(\mathbb{R}^3)}
	\sim
	\| \mathbf{g}_1\| _{ \mathbf{L}^2(\mathbb{S}^2)}
	+
	\| \mathbf{g}_2\| _{ \mathbf{L}^2(\mathbb{S}^2)}
	\sim
	\| \mathbf{u}_p\| _{ \mathcal{H}(\mathbb{R}^3)}
	+
	\| \mathbf{u}_s\| _{ \mathcal{H}(\mathbb{R}^3)},
	\end{equation*}
	where $\mathbf{g}_1$ and $\mathbf{g}_2,$ given in \eqref{elasticHerglotz}, are the longitudinal and transverse far-field patterns of $ \mathbf{u}$ respectively.
\end{theorem}
Throughout this paper, for non-negative quantities $a$ and $b$ we use $a\lesssim b$ ($a\gtrsim b$) to denote the existence of a positive numerical constant $c$ such that $a\le c\,b$ ($a\ge c\,b$). We write $a\sim b$ if both $a\le c\,b$ and $a\ge c\,b$.

We notice that this theorem is an extension of \cite[Theorem 1.1]{2d}  to the three dimensional context. However, this extension involves important technical difficulties that require different arguments and tools. Since the interesting results in elasticity concern only dimensions $2$ and $3$, we have not addressed higher dimensions, which will certainly entail
greater technical difficulties.

The interest of the characterization given in Theorem \ref{caracterizacion} is that
it provides a structure of a Hilbert space with reproducing kernel to the space of elastic Herglotz wave functions. Roughly speaking, this means that 
we will be able to construct a second-order tensor, the so called reproducing kernel, that reproduces in some way every elastic Herglotz wave function. %Since this construction is somehow a transform theory, it might simplify the understanding of certain problems in the context of elasticity. 

The study of reproducing kernels started at the beginning of the last century with Zaremba \cite{Zaremba1}-\cite{Zaremba2} in the context of boundary value problems for harmonic functions, and with Mercer \cite{Mercer} in the context of integral equations. Later on, Bergman \cite{Bergman} observed in his thesis the reproducing property of kernels built from orthogonal systems of harmonic and analytic functions in one and several variables. In the next years, a lot of important results were achieved by the use of these kernels and was Aronszajn \cite{Aronszajn} who systematically established a general theory. One of the main points of his work was to show the one to one correspondence between the class of reproducing kernels and the class of positive definite functions.

The reproducing kernel Hilbert spaces have many applications and recently there has been a new interest for them in different frameworks, like quantum mechanics, signal processing, probability theory and statistical learning theory.  We refer the reader to \cite{Berlinet}, \cite{Paulsen} and  the references there in for a complete account of it.

In the context of Herglotz wave functions,  \'Alvarez et al. 
characterized in \cite{A+F-G+P-E} the space of these functions in dimension two by certain Hilbert space, and proved that it is a reproducing kernel Hilbert space, calculating the reproducing kernel. Recently, P\'erez-Esteva and Valenzuela D\'iaz extended this result to higher dimensions in \cite{P-E+V-D}.

Following their ideas, we introduce the space $\mathcal{H}(\mathbb{R}^3)$ as the completion of the vector fields in $C^\infty_c(\mathbb{R}^3;\mathbb{C}^3)$ with respect to the norm $\|\cdot\|_{\mathcal{H}(\mathbb{R}^3)}$ given in \eqref{norma3d}. This is a Hilbert space with the inner product associated given by
\begin{equation}
\label{Producto}
\langle\mathbf{u},\mathbf{v}\rangle_{\mathcal{H}(\mathbb{R}^3)}:=
\int_{\mathbb{R}^{3}}
\left( \mathbf{u}(\bm{x})\cdot \overline{\mathbf{v}(\bm{x})}+\nabla_S\mathbf{u}(\bm{x}):\overline{\nabla_S\mathbf{v}(\bm{x})}\right) 
\frac{d\bm{x}}{\langle\bm{x}\rangle^3}.
\end{equation} 
Note that using polar coordinates we can rewrite
\begin{equation}\label{Producto2}
\langle {{\bf u},{\bf v}} \rangle_{\mathcal{H}(\mathbb{R}^3)}=\int_{0}^\infty \left(\langle {{\bf u},{\bf v}} \rangle_{L^2(\mathbb{S}^2)}+\langle {{\nabla_S {\bf u}},\nabla_S {\bf v}} \rangle_{L^2(\mathbb{S}^2)} \right)\,r^2\,\frac{dr}{\langle r\rangle^3},
\end{equation}
where
\begin{equation}\label{def:innerL2}
\langle {{\bf u},{\bf v}} \rangle_{L^2(\mathbb{S}^2)}=\int_{\mathbb{S}^2}{\bf u}(r\xi)\cdot\overline {{\bf v}(r\xi)}\,d \sigma(\xi)
\end{equation}
with $d\sigma$ denoting the measure induced by the Lebesgue measure over the unit sphere $\mathbb{S}^2$. 

It is important to observe that from Theorem \ref{caracterizacion}, the set of all three-dimensional elastic Herglotz wave functions,
that from now on we will denote by $\mathcal{W}(\mathbb{R}^3)$, is a closed subspace of $\mathcal{H}(\mathbb{R}^3). $ Moreover, we have the following result.

\begin{theorem}
	\label{NucleoReporductor3d}
	The set of all three-dimensional elastic Herglotz wave functions, $\mathcal{W}(\mathbb{R}^3)$, is a reproducing kernel Hilbert space with respect to the inner product given in \eqref{Producto}.
\end{theorem}

The analogous result to Theorem \ref{caracterizacion} for dimension 2, given in \cite[Theorem 1.1 ]{2d}, will allows us to prove the analogous to Theorem \ref{NucleoReporductor3d} for the two-dimensional case. 

In order to do that, we introduce the spaces $\mathcal{H}(\mathbb{R}^2)$ 
and $\mathcal{W}(\mathbb{R}^2),$ in the same way as we did for the three-dimensional case. From \eqref{LoMismo}, the inner product associated to he norm $\|\cdot\|_{\mathcal{H}(\mathbb{R}^2)}$ given in \eqref{norma2d} can be written as
\begin{equation}
\label{Producto2d}
\langle\mathbf{u},\mathbf{v}\rangle_{\mathcal{H}(\mathbb{R}^2)}:=
\int_{\mathbb{R}^{2}}
\left( \mathbf{u}(\bm{x})\cdot \overline{\mathbf{v}(\bm{x})}+\nabla_\varphi\mathbf{u}(\bm{x}):\overline{\nabla_\varphi\mathbf{v}(\bm{x})}\right) 
\frac{d\bm{x}}{\langle\bm{x}\rangle^3},
\end{equation}  
where $\nabla_\varphi$ is given by \eqref{NablaAngular2d}. 

We have that $\mathcal{H}(\mathbb{R}^2)$ is a Hilbert space with this inner product 
and $\mathcal{W}(\mathbb{R}^2)$ is a closed subspace of $\mathcal{H}(\mathbb{R}^2).$
\begin{theorem}
	\label{NucleoReporductor2d}
	The set of all two-dimensional elastic Herglotz wave functions, $\mathcal{W}(\mathbb{R}^2)$, is a reproducing kernel Hilbert space with respect to the inner product given in \eqref{Producto2d}.
\end{theorem} 

For shortness, throughout this paper we will write $\langle\cdot,\cdot\rangle_{\mathcal{H}}$ to denote both, $\langle\cdot,\cdot\rangle_{\mathcal{H}(\mathbb{R}^2)}$ and $\langle\cdot,\cdot\rangle_{\mathcal{H}(\mathbb{R}^3)}$, and we will use $\|\cdot\|_{\mathcal{H}}$ for the associated norms. The context will avoid any confusion.

The rest of the paper is organized as follows. The second section is devoted to the proof of the Theorem \ref{caracterizacion}, and the third one to the proofs of the Theorems \ref{NucleoReporductor3d} and \ref{NucleoReporductor2d}. Finally, the fourth section contains some technical lemmas needed in the previous sections.
\section{Characterization of $\mathcal{W}(\mathbb{R}^3)$}  \label{sec:1}
This section is devoted to the proof of Theorem \ref{caracterizacion}. We split it into two subsections. The first one contains the proof itself and certain auxiliary results. One of them (see Proposition \ref{ortoLMN} below) is the key point in the proof of the theorem. Since this proposition is quite technical, and requires of several
complementary lemmas, we postpone its proof to the second subsection.

\subsection{Proof of Theorem \ref{caracterizacion}}\label{sec:2}
As we said in \eqref{descomposicionHelmholtz}, any solution $\mathbf{u}$ of the spectral Navier equation \eqref{spectralNavier} can be written as the sum of its compressional part $\mathbf{u}_p$ and its shear part $\mathbf{u}_s$, which are given in \eqref{upus}.

On the other hand, the solenoidal character of the three-dimensional vector field $\mathbf{u}_s$ implies the existence of two scalar fields $g$ and $h$ such that
\begin{equation}
\label{poloidal}
\mathbf{u}_s(\bm{x})
=\nabla\times(\bm{x}g(\bm{x}))
+\nabla\times(\nabla\times(\bm{x}h(\bm{x}))).
\end{equation}
Moreover, the wave potentials $g$ and $h$ are solutions of the scalar homogeneous Helmholtz equation with wave number $k_s.$ 

Combining \eqref{descomposicionHelmholtz},\eqref{upus} and \eqref{poloidal} we have the following representation for the time-independent displacement field:
\begin{equation}
\label{lau}
\mathbf{u}(\bm{x})
=
\nabla f(\bm{x})
+
\nabla\times(\bm{x}g(\bm{x}))
+
\nabla\times(\nabla\times(\bm{x}h(\bm{x}))),
\end{equation}
where the wave potentials $f,g$ and $h$ satisfy the scalar Helmholtz equations
\begin{eqnarray*}
	\Delta f+k_p^2f&=&0,\\
	\Delta g+k_s^2g&=&0,\\
	\Delta h+k_s^2h&=&0.\\
\end{eqnarray*}
It is well known (see \cite[Chapter IV]{Stein-Weiss}) that every entire solution $v$ of the scalar Helmholtz equation $\Delta v+k^2 v=0$ with wave number $k>0$ in dimension three, can be expanded as 
\begin{equation}  
\label{Desarrollo}
\displaystyle v(\bm{x})=\sum_{\ell=0}^\infty\sum_{m=-\ell}^\ell a_{\ell,m}\,F_{\ell,k}^m(\bm{x}),
\end{equation}
where $a_{\ell,m}$ are certain constants and, using the spherical coordinates given in \eqref{CoordenadasEsfericas},
\begin{equation}\label{def: F}
F_{\ell,k}^m(\bm{x})=j_{\ell}(k\,r)Y_\ell^m(\theta,\varphi),
\end{equation}
with 
$j_{\ell}$ denoting the spherical Bessel function of the first kind and 
$Y_\ell^m$ is the spherical harmonic function of degree $\ell$ and order $m$.

Motivated by the expansion \eqref{Desarrollo} and the decomposition given in \eqref{lau}, we introduce the so called Navier eigenvectors or Hansen vectors (see \cite{Hansen} or \cite{Ben}) given by
\begin{eqnarray}
\mathbf{L}_\ell^m({\bm x})&:=&\nabla(F_{\ell,k_p}^{m}({\bm x})),
\label{def:L}
\\
\mathbf{M}_\ell^m({\bm x})&:=&\nabla\times({\bm x}F_{\ell,k_s}^{m}({\bm x})),
\label{def:M}\\
\mathbf{N}_\ell^m({\bm x})&:=& \nabla\times(\nabla\times({\bm x} F_{\ell,k_s}^{m}({\bm x}))).
\label{def:N}
\end{eqnarray}
For convenience, we introduce the corresponding unitary vectors with respect to the norm given in \eqref{norma3d}, that is, 
\begin{equation}\label{def:NNvectors}
\mathbfcal{L}_\ell^m({\bm x}):=\frac{{\mathbf{L}}_\ell^m({\bm x})}{\| {\mathbf{L}}_\ell^m\|_{\mathcal H}},\,\mathbfcal{M}_\ell^m(\bm x):=\frac{\mathbf{M}_\ell^m({\bm x})}{\|\mathbf{M}_\ell^m\|_{\mathcal H}},\,\mathbfcal{N}_\ell^m(\bm x):=\frac{\mathbf{N}_\ell^m({\bm x})}{\| \mathbf{N}_\ell^m\|_{\mathcal H}}.
\end{equation}
With this notation, we have that 
\begin{equation}
\label{DesarrolloLame}
\displaystyle \mathbf{u}(\bm{x})=\sum_{\ell=0}^\infty\sum_{m=-\ell}^\ell 
\left(
a_{\ell,m}\mathbfcal{L}_\ell^m({\bm x})
+
b_{\ell,m}\mathbfcal{M}_\ell^m({\bm x})
+
c_{\ell,m}\mathbfcal{N}_\ell^m({\bm x})
\right),
\end{equation}
for certain constants $a_{\ell,m},b_{\ell,m},c_{\ell,m}.$
Moreover, we have that the compressional and the shear parts of $\mathbf{u}$ can be written as
\begin{eqnarray}
\label{Desarrolloup}
\displaystyle \mathbf{u}_p(\bm{x})&=&\sum_{\ell=0}^\infty\sum_{m=-\ell}^\ell 
a_{\ell,m}\mathbfcal{L}_\ell^m({\bm x}),
\\
\label{Desarrollous}
\displaystyle \mathbf{u}_s(\bm{x})&=&\sum_{\ell=0}^\infty\sum_{m=-\ell}^\ell 
\left(
b_{\ell,m}\mathbfcal{M}_\ell^m({\bm x})
+
c_{\ell,m}\mathbfcal{N}_\ell^m({\bm x})
\right).
\end{eqnarray} 
The following lemma will help us to derive some convergence results for the series given in \eqref{Desarrollo},\eqref{DesarrolloLame}--\eqref{Desarrollous}.
\begin{lemma}
	\label{ConvergenciaMadre}
	Let $F_{\ell,k}^m$ be the function defined in \eqref{def: F}. The series
	\begin{equation}
	\label{Madre}
	\sum_{\ell=0}^\infty\sum_{m=-\ell}^\ell |F_{\ell,k}^m(\bm{x})|^2
	\end{equation}
	converges absolutely and uniformly on compact subsets of $\mathbb{R}^3.$ 
	
	Moreover, the result holds if we replace in \eqref{Madre} the function $F_{\ell,k}^m(\bm{x})$ by any derivative of any order of the function.
\end{lemma}
\begin{proof}
	From \eqref{def: F}, using the addition theorem for spherical harmonics that states
	\begin{equation}
	\label{Adittion}
	\sum_{m=-\ell}^\ell |Y_{\ell}^m(\theta,\varphi)|^2=\frac{2\ell+1}{4\pi},
	\end{equation}
	one can see that to prove the convergence of \eqref{Madre}, it is enough to prove that the series 
	\begin{equation}
	\label{ConvergeBessel}
	\sum_{\ell=0}^\infty \ell\, \left(j_{\ell}(r)\right)^2
	\end{equation}
	converges absolutely and uniformly on compact subsets of $(0,\infty).$ This can be proved, using the well known identity 
	\begin{equation}
	\label{Jj}
	j_{\ell}(r)=\sqrt{\frac{\pi}{2r}}\,J_{\ell+\frac{1}{2}}(r),
	\end{equation}
	where $J_\mu$ denotes the ordinary Bessel functions of the first kind and order $\mu,$ and the following asymptotic expansion of $J_\mu$ (see 9.3.1 in \cite{abramowitz})
	\begin{equation}
	\label{CotaBessel}
	J_\mu(r)\sim \frac{1}{\mu^{\frac{1}{2}}}\left(\frac{e\,r}{2\,\mu}\right)^\mu,\qquad r>0,\,\mu\rightarrow\infty.
	\end{equation}
	To prove the result for first order derivatives of $F_{\ell,k}^m(\bm{x})$, if we  
	use the spherical coordinates given in \eqref{CoordenadasEsfericas} and the gradient operator defined in \eqref{gradiente} we have that
	\begin{equation}
	\label{GradienteF}
	\nabla F_{\ell,k}^m(\bm{x})=\widehat{\bm{r}}\,kj'_\ell(kr)Y_\ell^m(\theta,\varphi)
	+
	\widehat{\bm{\theta}}\,\frac{j_\ell(kr)}{r}\,\frac{\partial Y_\ell^m}{\partial \theta}(\theta,\varphi)
	+
	\widehat{\bm{\varphi}}\,j_\ell(kr)\frac{\frac{\partial Y_\ell^m}{\partial \varphi}(\theta,\varphi)}{\sin\theta}.
	\end{equation}
	Then using the identities 
	\begin{equation}
	\label{jDerivada}
	j'_\ell(r)=\frac{\ell\, j_{\ell-1}(r)-(\ell+1)\,j_{\ell+1}(r)}{2\ell+1},
	\end{equation}
	\begin{equation}
	\label{jyr}
	\frac{j_\ell(kr)}{r}=\frac{ j_{\ell-1}(r)+\,j_{\ell+1}(r)}{2\ell+1},
	\end{equation}
	\begin{equation*}
	\label{Ytheta}
	\frac{\partial Y_\ell^m}{\partial \theta}(\theta,\varphi)=m\cot\theta\, Y_\ell^m(\theta,\varphi)
	+
	\sqrt{(\ell-m)(\ell+m+1)}e^{-i\varphi}Y_\ell^m(\theta,\varphi),
	\end{equation*}
	\begin{equation*}
	\label{Yvarphi}
	\frac{\partial Y_\ell^m}{\partial \varphi}(\theta,\varphi)=im\,Y_\ell^m(\theta,\varphi),
	\end{equation*}
	and \eqref{Adittion} in \eqref{GradienteF}, 
	one gets the result from the convergence of the series given in \eqref{ConvergeBessel}.
	
	Arguing as before, we get the result for second order derivatives of $F_{\ell,k}^m(\bm{x})$, and so on.
\end{proof}
\begin{remark}
	\label{ConvergenciaSerie}
	Let $v$ be a Herglotz wave function, then the series given in \eqref{Desarrollo} converges absolutely and uniformly on compact subsets of $\mathbb{R}^3.$ Moreover, we can differentiate it repeatedly term by term with absolute and uniform convergence in compact subsets of $\mathbb{R}^3.$
	
	This is due to the Cauchy-Schwarz inequality, the fact that 
	$$
	\sum_{\ell=0}^\infty\sum_{m=-\ell}^\ell |a_{\ell,m}|^2<\infty,
	$$
	and Lemma \ref{ConvergenciaMadre}.
\end{remark}
\begin{remark}
	\label{convergenciaSerieLame}
	Let ${\bf u}$ be a elastic Herglotz wave function, then the series given in \eqref{DesarrolloLame}--\eqref{Desarrollous} converge absolutely and uniformly on compact subsets of $\mathbb{R}^3.$ Moreover, they can be differentiated term by term, with absolute and uniform convergence in compact subsets of $\mathbb{R}^3.$
	
	Using the triangle inequality, the Cauchy-Schwarz inequality, and the fact that
	$$
	\sum_{\ell=0}^\infty\sum_{m=-\ell}^\ell |a_{\ell,m}|^2<\infty,
	\quad
	\sum_{\ell=0}^\infty\sum_{m=-\ell}^\ell |b_{\ell,m}|^2<\infty,
	\quad
	\sum_{\ell=0}^\infty\sum_{m=-\ell}^\ell |c_{\ell,m}|^2<\infty,
	$$
	it is enough to prove the result for the following series:
	\begin{equation}
	\label{series}
	\sum_{\ell=0}^\infty\sum_{m=-\ell}^\ell |\mathbfcal{L}_\ell^m({\bm x})|^2,
	\quad
	\sum_{\ell=0}^\infty\sum_{m=-\ell}^\ell |\mathbfcal{M}_\ell^m({\bm x})|^2,
	\quad
	\sum_{\ell=0}^\infty\sum_{m=-\ell}^\ell |\mathbfcal{N}_\ell^m({\bm x})|^2.
	\end{equation}
	
	Actually, we can replace in \eqref{series} the unitary Hansen vectors by the Hansen vectors defined in \eqref{def:L}--\eqref{def:N}, because, as we will prove later on (see Proposition \ref{ortoLMN} bellow), we have that
	\begin{equation}
	\label{NormasUnitarios}
	\|\mathbf{L}_\ell^m\|_{\mathcal{H}}=O(1),\qquad
	\|\mathbf{N}_\ell^m\|_{\mathcal{H}}=\|\mathbf{M}_\ell^m\|_{\mathcal{H}}=O(\ell).
	\end{equation}
	
	Since the definition of the Hansen vectors are given in terms of derivatives of the function $F_{\ell,k}^m$ defined in \eqref{def: F}, the convergence result is a consequence of Lemma \ref{ConvergenciaMadre}.
\end{remark}
The key point to prove Theorem \ref{caracterizacion} is the \emph{almost} orthogonality property of the family $\{\mathbf{L}_\ell^m, \mathbf{M}_\ell^m,\mathbf{N}_\ell^m\}_{\ell,m}$ with respect to the inner product \eqref{Producto}. The following proposition shows such property.
\begin{proposition}	\label{ortoLMN}
	Let $\ell,\ell',m,m'\in\mathbb{Z}$ such that $\ell,\ell'\geq 0$, $m\in[-\ell, \ell]$ and $m'\in[-\ell', \ell']$. The Navier eigenvectors defined by \eqref{def:L}--\eqref{def:N} satisfy that
	\begin{align}
	\langle {\mathbf{L}_\ell^m,\mathbf{L}_{\ell'}^{m'}}\rangle_{\mathcal{H}} &=O(1)\delta_{mm'}\delta_{\ell\ell'},
	\qquad \ell\rightarrow\infty,\nonumber\\%\label{ort-norma-lmn-0}
	\langle {\mathbf{M}_\ell^m,\mathbf{M}_{\ell'}^{m'}} \rangle_{\mathcal{H}}&=O(\ell^2)\delta_{mm'}\delta_{\ell\ell'},
	\qquad \ell\rightarrow\infty,\nonumber\\%\label{ort-norma-lmn-1}
	\langle {\mathbf{N}_\ell^m,\mathbf{N}_{\ell'}^{m'}} \rangle_{\mathcal{H}}&=O(\ell^2)\delta_{mm'}\delta_{\ell\ell'},
	\qquad \ell\rightarrow\infty,\nonumber\\%\label{ort-norma-lmn-1b}
	\langle {\mathbf{L}_\ell^m,\mathbf{M}_{\ell'}^{m'}} \rangle_{\mathcal{H}}&=\langle {\mathbf{M}_\ell^m,\mathbf{N}_{\ell'}^{m'}} \rangle_{\mathcal{H}}=0,
	\qquad \ell\rightarrow\infty,\nonumber\\%\label{ort-norma-lmn-2}
	\langle {\mathbf{L}_\ell^m,\mathbf{N}_{\ell'}^{m'}} \rangle_{\mathcal{H}}&=o(1)\,\delta_{\ell\ell'}\delta_{mm'},
	\qquad \ell\rightarrow\infty,\label{ort-norma-lmn-3}
	\end{align}
	where $\delta$ denotes the Kronecker delta and $\langle\cdot,\cdot\rangle_{\mathcal{ H}}$ is the inner product defined in \eqref{Producto}.
\end{proposition}
%%%%%%%%%%%%%%%%%%%%%%%%%%%%%%%%%%%%%%%%%%%%%%%%%%%%%%%%%%%%%%%%%%%Teresa trabajando

\emph{Proof of Theorem  \ref{caracterizacion}.} 
We first prove the necessary condition. Since $\bf u\in\mathcal{W}$, from \eqref{HerglotzCondition}, we have that $\|{\bf u}\|_A<\infty.$ Moreover, using \eqref{upusfinita}, it follows that $\|{\bf u_p}\|_A$, $\|{\bf u_s}\|_A<\infty$, where ${\bf u_p}$ and ${\bf u_s}$ are defined in \eqref{upus}. Since ${\bf u_p}$ and ${\bf u_s}$ are solutions of the vectorial Helmholtz equations \eqref{eqcompressional} and \eqref{eqshear} respectively, we have from \eqref{HerglotzConditionHW} that its components are Herglotz wave functions. Using the characterization of the Herglotz wave functions given in \eqref{salvador}, we have that $\|{\bf u_p}\|_{\mathcal{H}}$, $\|{\bf u_s}\|_{\mathcal{H}}<\infty.$ From here, since ${\bf u}={\bf u_p}+{\bf u_s}$, we get $\|{\bf u}\|_{\mathcal{H}}<\infty.$

For the sufficient condition, we consider that $\|\bf u\|_{\mathcal{ H}}<\infty.$ Taking the expansion \eqref{DesarrolloLame}, the definition of the inner product \eqref{Producto} and the orthogonality properties in Proposition \ref{ortoLMN} we have that
\begin{align}\label{suma1}
\|{\bf u}\|_{\mathcal{ H}}^2&\geq\sum_{\ell=0}^\infty\sum_{m=-\ell}^\ell \Big(|a_{\ell,m}|^2+|b_{\ell,m}|^2+|c_{\ell,m}|^2\Big)\nonumber\nonumber\\
&- \sum_{\ell=0}^\infty\sum_{m=-\ell}^\ell|a_{\ell,m}||c_{\ell,m}|\langle\mathbfcal{L}_{\ell}^m,\mathbfcal{N}_{\ell}^m\rangle_{\mathcal{H}}|.
\end{align}
We notice that, on one hand, \eqref{NormasUnitarios} and \eqref{ort-norma-lmn-3} assures that there exists $\ell_0\in\mathbb{N}$ such that
\begin{equation}\label{estima1t1}
|\langle\mathbfcal{L}_{\ell}^m,\mathbfcal{N}_{\ell}^m\rangle_{\mathcal{H}}|<\frac{1}{2},\quad\forall \ell>\ell_0,\quad m\in\mathbb{Z}\cap[-\ell,\ell].
\end{equation}
On the other hand, observe that 
\begin{equation}\label{LNdif}
|\langle\mathbfcal{L}_{\ell}^m,\mathbfcal{N}_{\ell}^m\rangle_{\mathcal{H}}|\neq 1,\quad\forall \ell\in\mathbb{Z}^+,\quad m\in\mathbb{Z}\cap[-\ell,\ell],
\end{equation}
since in other case $\mathbfcal{N}_{\ell}^m=c\mathbfcal{L}_{\ell}^m$ and using \eqref{def:NNvectors}, \eqref{def:L} and \eqref{def:N} we would conclude that 
$\Delta F_{\ell,k_p}^{m}=0,$ which is false. Then from here, we have that there exists $0<c<1$ such that
\begin{equation}\label{estima2t1}
|\langle\mathbfcal{L}_{\ell}^m,\mathbfcal{N}_{\ell}^m\rangle_{\mathcal{H}}|<c,\quad\forall \ell\leq\ell_0,\quad m\in\mathbb{Z}\cap[-\ell,\ell].
\end{equation}
Then using Cauchy-Schwarz, \eqref{estima1t1} and \eqref{estima2t1}, we have that
\begin{align}
\sum_{\ell=0}^\infty\sum_{m=-\ell}^\ell |a_{\ell,m}||c_{\ell,m}||\langle\mathbfcal{ L}_{\ell}^m,\mathbfcal{ N}_{\ell}^m\rangle_{\mathcal{H}}|&\leq c\left(\sum_{\ell\leq\ell_0}\sum_{m=-\ell}^\ell |a_{\ell,m}|^2\right)^{\frac{1}{2}}\left(\sum_{\ell\leq \ell_0}\sum_{m=-\ell}^\ell |c_{\ell,m}|^2\right)^{\frac{1}{2}}\nonumber\\
&+\frac{1}{2}\left(\sum_{\ell>\ell_0}\sum_{m=-\ell}^\ell |a_{\ell,m}|^2\right)^{\frac{1}{2}}\left(\sum_{\ell>\ell_0}\sum_{m=-\ell}^\ell |c_{\ell,m}|^2\right)^{\frac{1}{2}}\nonumber\\
&<\frac{1}{2}\sum_{\ell=0}^\infty\sum_{m=-\ell}^\ell (|a_{\ell,m}|^2+ |c_{\ell,m}|^2)\label{estima3t1},
\end{align}
where in the last inequality we have used the fact that $2ab\leq(a^2+b^2)$. Then using \eqref{estima3t1} in \eqref{suma1} together with identities \eqref{def:NNvectors}, \eqref{Desarrolloup},\eqref{Desarrollous} and Proposition \ref{ortoLMN}, we obtain
\[\|{\bf u}\|_{\mathcal{ H}}^2\gtrsim\|{\bf u}_p\|_{\mathcal{ H}}^2+\|{\bf u}_s\|_{\mathcal{ H}}^2.\]
Then by hypothesis, we can assure that $\|{\bf u}_p\|_{\mathcal{ H}}, \|{\bf u}_s\|_{\mathcal{ H}}<\infty$. Because ${\bf u}_p$ and ${\bf u}_s$ are solutions of the vectorial Helmholtz equations \eqref{eqcompressional} and \eqref{eqshear} respectively, using the characterization of the Herglotz wave functions given in \eqref{salvador}, we have that its components are Herglotz wave functions. Therefore, by \eqref{HerglotzConditionHW}, we have that $\|{\bf u}_p\|_{A}, \|{\bf u}_s\|_{A}<\infty$ and then, \eqref{upusfinita} implies the required conclusion.
\hfill $\Box$

%%%%%
\subsection{Proof of Proposition \ref{ortoLMN}}\label{sub:sec21}
To prove  Proposition \ref{ortoLMN} we need the following technical lemmas
\begin{lemma}{\rm(\cite[pp.60 ]{Ben})}\label{lema_auxiliarLMN}%Identity (2.86)
	Let $\ell,\ell',m,m'\in\mathbb{Z}$ such that $\ell,\ell'\geq 0$, $m\in[-\ell, \ell]$ and $m'\in[-\ell', \ell']$. The Navier eigenvectors defined by \eqref{def:L}--\eqref{def:N} satisfy the following identities.
	\begin{align*}
	\langle {\mathbf{L}_\ell^m,\mathbf{L}_{\ell'}^{m'}}\rangle_{L^2(\mathbb{S}^2)}&=\frac{k_p^2}{2\ell+1}\Big(\ell |j_{\ell-1}(k_p\,r)|^2+(\ell+1)|j_{\ell+1}(k_p\,r)|^2\Big)\,\delta_{mm'}\delta_{\ell \ell '},\\%\label{S2-lmn-0}
	\langle {\mathbf{M}_\ell^m,\mathbf{M}_{\ell'}^{m'}} \rangle_{L^2(\mathbb{S}^2)}&=\ell(\ell+1)|j_{\ell}(k_s\,r)|^2\,\delta_{mm'}\delta_{\ell \ell '},\\%\label{S2-lmn-1}
	\langle {\mathbf{N}_\ell^m,\mathbf{N}_{\ell'}^{m'}} \rangle_{L^2(\mathbb{S}^2)}&=\frac{k_s^2\ell(\ell+1)}{2\ell+1}\Big((\ell+1) |j_{\ell-1}(k_s\,r)|^2+\ell|j_{\ell+1}(k_s\,r)|^2\Big)\,\delta_{mm'}\delta_{\ell \ell '},\\%\label{S2-lmn-1b}
	\langle {\mathbf{L}_\ell^m,\mathbf{M}_{\ell'}^{m'}} \rangle_{L^2(\mathbb{S}^2)}&=\langle {\mathbf{M}_\ell^m,\mathbf{N}_{\ell'}^{m'}} \rangle_{L^2(\mathbb{S}^2)}=0,\\%\label{S2-lmn-2}
	\langle {\mathbf{L}_\ell^m,\mathbf{N}_{\ell'}^{m'}} \rangle_{L^2(\mathbb{S}^2)}&=\frac{k_p\,k_s\,\ell(\ell+1)}{2\ell+1}\Big( j_{\ell-1}(k_p\,r)j_{\ell-1}(k_s\,r)\\
	&-j_{\ell+1}(k_p\,r)j_{\ell+1}(k_s\,r)\Big)\,\delta_{mm'}\delta_{\ell \ell '},%\label{S2-lmn-3}
	\end{align*}
	where $\delta$ denotes the Kronecker delta and $\langle\cdot,\cdot\rangle_{L^2(\mathbb{S}^2)}$ is the inner product defined in \eqref{def:innerL2}.
\end{lemma}
\begin{lemma}\label{lema_auxiliaresfericosLMN}
	Let $\ell,\ell',m,m'\in\mathbb{Z}$ such that $\ell,\ell'\geq 0$, $m\in[-\ell, \ell]$ and $m'\in[-\ell', \ell']$. The Navier eigenvectors defined by \eqref{def:L}--\eqref{def:N} satisfy the following identities.
	\begin{align*}%\label{S2-esf-lmn-0}
	\langle {\nabla_S\mathbf{L}_\ell^m,\nabla_S\mathbf{L}_{\ell'}^{m'}}\rangle_{L^2(\mathbb{S}^2)}&=\frac{k_p^2}{2\ell+1}\Big(\ell^2(\ell-1) |j_{\ell-1}(k_p\,r)|^2\\
	&+(\ell+1)^2(\ell+2)|j_{\ell+1}(k_p\,r)|^2\Big)\,\delta_{mm'}\delta_{\ell \ell '},\\
	\langle {\nabla_S\mathbf{M}_\ell^m,\nabla_S\mathbf{M}_{\ell'}^{m'}} \rangle_{L^2(\mathbb{S}^2)}&=\ell^2(\ell+1)^2|j_{\ell}(k_s\,r)|^2\,\delta_{mm'}\delta_{\ell \ell '},\\%\label{S2-esf-lmn-1}
	\langle {\nabla_S\mathbf{N}_\ell^m,\nabla_S\mathbf{N}_{\ell'}^{m'}} \rangle_{L^2(\mathbb{S}^2)}&=\frac{k_s^2\,\ell^2(\ell+1)^2}{2\ell+1}\Big((\ell-1) |j_{\ell-1}(k_s\,r)|^2\\%\label{S2-esf-lmn-1b}
	&+(\ell+2)|j_{\ell+1}(k_s\,r)|^2\Big)\,\delta_{mm'}\delta_{\ell \ell '},\\
	\langle {\nabla_S\mathbf{L}_\ell^m,\nabla_S\mathbf{M}_{\ell'}^{m'}} \rangle_{L^2(\mathbb{S}^2)}&=\langle {\nabla_S\mathbf{M}_\ell^m,\nabla_S\mathbf{N}_{\ell'}^{m'}} \rangle_{L^2(\mathbb{S}^2)}=0,\\%\label{S2-esf-lmn-2}
	\langle {\nabla_S\mathbf{L}_\ell^m,\nabla_S\mathbf{N}_{\ell'}^{m'}} \rangle_{L^2(\mathbb{S}^2)}&=\frac{k_p\,k_s\,\ell(\ell+1)}{2\ell+1}\Big( \ell(\ell-1)\,j_{\ell-1}(k_p\,r)j_{\ell-1}(k_s\,r)\nonumber\\
	&-(\ell+2)(\ell+1)j_{\ell+1}(k_p\,r)j_{\ell+1}(k_s\,r)\Big)\,\delta_{mm'}\delta_{\ell \ell '},%\label{S2-esf-lmn-3}
	\end{align*}
	where $\delta$ denotes the Kronecker delta, $\nabla_S$ is the spherical gradient introduced in \eqref{NablaAngular3d}, and $\langle\cdot,\cdot\rangle_{L^2(\mathbb{S}^2)}$ is the inner product defined in \eqref{def:innerL2}.
\end{lemma}
\begin{lemma}\label{integralesBessel1}
	Let $a,b>0$ and $a\neq b$. For $\langle r\rangle:=(1+r^2)^{1/2}$, the following identities hold.
	\begin{equation}\label{b}
	\int_0^\infty |j_\ell(r)|^2 \,r^2\,\frac{dr}{\langle r\rangle^3}=O(\ell^{-2}),\qquad \ell\rightarrow\infty.
	\end{equation}	
	\begin{equation}\label{a}
	\left|\int_0^\infty j_\ell(ar)j_\ell(br) \,r^2\,\frac{dr}{\langle r\rangle^3}\right|=O(\delta^\ell\,\ell^{-\frac{3}{2}}),\qquad \ell\rightarrow\infty,
	\end{equation}
	with $\delta=\min\left(\frac{a}{b},\frac{b}{a}\right)$.
\end{lemma}
The proof of \eqref{b} is a consequence of the asymptotics in \cite[pp.661]{juan} and the proof can be found in (6) of \cite{stempak}, taking $p=2$ and $\alpha=-1$. The proof of \eqref{a} is detailed at the end of the subsection.

\medskip
\emph{Proof of Proposition \ref{ortoLMN}.}
Using the expression of the inner product given in \eqref{Producto2}, the proof follows directly from the previous Lemma \ref{lema_auxiliarLMN}, Lemma \ref{lema_auxiliaresfericosLMN} and Lemma \ref{integralesBessel1}.
\hfill$\Box$

\medskip

%%%%%
It remains to prove Lemma \ref{lema_auxiliaresfericosLMN} and \eqref{a}. To prove the lemma, we rewrite the Navier eigenvectors $\mathbf{L}_\ell^m$, $\mathbf{M}_\ell^m$ and $\mathbf{N}_\ell^m$ separating the radial and the spherical parts. Since the spherical harmonic $Y_\ell^m$ is given by
\begin{equation*}%\label{def:Yl}
Y_\ell^m(\theta,\varphi)=\frac{1}{\sqrt{\Omega_{\ell m}}}{P}_\ell^{|m|}(cos\theta)e^{im\varphi},
\end{equation*}
where $P_\ell^m$ is the associated Legendre function and 
\begin{equation}\label{normalizing_constant}
\Omega_{\ell m}=\frac{4\pi}{2\ell+1}\frac{(\ell+|m|)!}{(\ell-|m|)!}
\end{equation}
is a normalizing constant, using identities \eqref{def: F} and \eqref{GradienteF}, it is not difficult to see that we can represent the Navier eigenvectors \eqref{def:L}--\eqref{def:N} as
\begin{align}
\mathbf{L}_\ell^m({\bm x})&=\left(k_pj_{\ell}'(k_p\,r){\bf P}_\ell^m(\theta,\varphi)+\big(\ell(\ell+1)\big)^{1/2}\,\frac{j_{\ell}(k_p\,r)}{r}{\bf B}_\ell^m(\theta,\varphi)\right),\label{redef:L}\\
\mathbf{M}_\ell^m({\bm x})&=\big(\ell(\ell+1)\big)^{1/2}j_{\ell}(k_s\,r){\bf C}_\ell^m(\theta,\varphi),\label{redef:M}\\
\mathbf{N}_\ell^m({\bm x})&=\ell(\ell+1)\frac{j_{\ell}(k_s\,r)}{r}{\bf P}_\ell^m(\theta,\varphi)+\big(\ell(\ell+1)\big)^{1/2}\left(k_sj_{\ell}'(k_sr)\right.\nonumber\\
&+\left.\frac{j_{\ell}(k_sr)}{r}\right){\bf B}_\ell^m(\theta,\varphi),\label{redef:N}
\end{align}
with
\begin{align}
{\bf P}_\ell^m(\theta,\varphi)&:=\Omega^{-\frac{1}{2}}_{\ell m}e^{im\varphi}P_\ell^{|m|}(\cos\theta){\bf\hat{r}},\label{def:Pl}\\
{\bf B}_\ell^m(\theta,\varphi)&:=\frac{\Omega^{-\frac{1}{2}}_{\ell m}e^{im\varphi}}{[\ell(\ell+1)]^{\frac{1}{2}}}\left(-\sin\theta 
(P_\ell^{|m|})'(\cos\theta){\hat{\bm\theta}}+\frac{im}{\sin\theta}P_\ell^{|m|}(\cos\theta) {\hat{\bm\varphi}}\right),\label{def:Bl}
\end{align}
and
\begin{align}\label{def:Cl}
{\bf C}_\ell^m(\theta,\varphi)&
:=\frac{\Omega^{-\frac{1}{2}}_{\ell m}e^{im\varphi}}{[\ell(\ell+1)]^{\frac{1}{2}}}\left(\frac{im}{\sin\theta}P_\ell^{|m|}(\cos\theta) {\hat{\bm\theta}} +\sin\theta(P_\ell^{|m|})'(\cos\theta){\hat{\bm\varphi}}\right).
\end{align}
Vectors \eqref{def:Pl}--\eqref{def:Cl} are known as \emph{Hansen harmonics} (see \cite{Hansen} and \cite[pp.54]{Ben} for more details on this definition). They verify the following orthogonal properties that are required in the proof of Lemma \ref{lema_auxiliaresfericosLMN}.
\begin{lemma}\label{lemma:ort_sphericalgrad}Let $\ell,\ell',m,m'\in\mathbb{Z}$ such that $\ell,\ell'\geq 0$, $m\in[-\ell, \ell]$ and $m'\in[-\ell', \ell']$. For the functions ${\bf P}_\ell^m$, ${\bf B}_\ell^m$ and ${\bf C}_\ell^m$, defined by \eqref{def:Pl}, \eqref{def:Bl} and \eqref{def:Cl} respectively, we have that
	\begin{align}
	\langle {\nabla_S{\bf P}_\ell^m,\nabla_S{\bf P}_{\ell'}^{m'}} \rangle_{L^2(\mathbb{S}^2)}&=\big(\ell(\ell+1)+2\big)\delta_{mm'}\,\delta_{\ell\ell'},\label{ort_grad-norma1}\\
	\langle {\nabla_S{\bf B}_\ell^m,\nabla_S{\bf B}_{\ell'}^{m'}} \rangle_{L^2(\mathbb{S}^2)}
	&=\langle {\nabla_S{\bf C}_\ell^m,\nabla_S{\bf C}_{\ell'}^{m'}} \rangle_{L^2(\mathbb{S}^2)}=\ell(\ell+1)\,\delta_{mm'}\,\delta_{\ell\ell'},\label{ort_grad-norma2}\\
	\langle {\nabla_S{\bf P}_\ell^m,\nabla_S{\bf C}_{\ell'}^{m'}} \rangle_{L^2(\mathbb{S}^2)}&=	\langle {\nabla_S{\bf B}_\ell^m,\nabla_S{\bf C}_{\ell'}^{m'}} \rangle_{L^2(\mathbb{S}^2)}=0,\label{ort_grad-norma3}\\
	\langle {\nabla_S{\bf P}_\ell^m,\nabla_S{\bf B}_{\ell'}^{m'}} \rangle_{L^2(\mathbb{S}^2)}&=-2\big(\ell(\ell+1)\big)^{1/2}\,\delta_{mm'}\,\delta_{\ell\ell'},\label{ort_grad-norma4}
	\end{align}
	where $\nabla_S$ is the spherical gradient introduced in \eqref{NablaAngular3d} and $\delta$ denotes the Kronecker delta.
\end{lemma} 

\emph{Proof of Lemma \ref{lema_auxiliaresfericosLMN}.}
Using the expressions for the Navier eigenvectors \eqref{redef:L}--\eqref{redef:N}, the proof can be derived directly from Lemma \ref{lemma:ort_sphericalgrad} together with the relations for the Bessel functions introduced in \eqref{jDerivada} and \eqref{jyr}.
\hfill$\Box$
\medskip

We give the proof of Lemma \ref{lemma:ort_sphericalgrad} that it is quite tecnical.
\medskip
\noindent

\emph{Proof of Lemma \ref{lemma:ort_sphericalgrad}.}
We will prove the identities one by one and for $m\geq 0$. The case $m<0$ is analogous.

\noindent
\underline{Proof of \eqref{ort_grad-norma1}}. 
From the definitions of ${\bf P}_\ell^m$ and $\nabla_S$ given in \eqref{def:Pl} and \eqref{NablaAngular3d} respectively, we obtain that
\begin{align}
\nabla_S{\bf P}_\ell^m(\theta,\varphi)= 
&
\Omega^{-\frac{1}{2}}_{\ell m}e^{im\varphi}
\left[
-\sin\theta (P_\ell^{m})'(\cos\theta){\hat {\bm \theta}}\otimes{\bf \hat r}+P_\ell^{m}(\cos\theta){\hat {\bm \theta}}\otimes{\hat {\bm\theta}}
\right. \nonumber
\\
&
\left.+\frac{im}{\sin\theta}P_\ell^{m}(\cos\theta){\hat{\bm\varphi}}\otimes{\bf \hat r}+P_\ell^{m}(\cos\theta){\hat{\bm\varphi}}\otimes{\hat{\bm\varphi}}
\right].
\label{def:gradP}
\end{align}
From here, we parameterize the sphere and make the change of variable $\cos\theta=x$, to get
\begin{align}
\langle {\nabla_S{\bf P}_\ell^m,\nabla_S{\bf P}_{\ell'}^{m'}} \rangle_{L^2(\mathbb{S}^2)}=
&
\int_0^\pi\left(\int_0^{2\pi}\nabla_S{\bf P}_\ell^m:\overline{\nabla_S{\bf P}_{\ell'}^{m'}}\,d\varphi\right)\sin\theta\, d\theta 
\nonumber
\\
=&
\,\Omega\,\delta_{mm'}
\left[
\int_{-1}^1(1-x^2)(P_{\ell }^{m})'(x)(P_{\ell' }^{m})'(x)\,dx
\right.
\nonumber
\\
&+
\left.
\int_{-1}^1
\left(2+\frac{m^2}{(1-x^2)}\right)P_{\ell }^{m}(x)P_{\ell' }^{m}(x)\,dx
\right]
\label{gradP_gradP}
\end{align}
where 
\begin{equation}\label{Omega}
\Omega:=\Omega(\ell,\ell',m)
=2\pi\Omega^{-\frac{1}{2}}_{\ell m}\Omega^{-\frac{1}{2}}_{\ell' m},
\end{equation}
with $\Omega_{\ell m}$ defined in \eqref{normalizing_constant}. 

We notice that the integrals appearing in \eqref{gradP_gradP} are definite integrals for any $m$, since
\begin{equation}
P_\ell^m(x)=(-1)^m(1-x^2)^{m/2}\frac{d^{m}P_{\ell}}{dx^{m}}(x),\quad{m\geq 0},
\label{Legendre}
\end{equation}
\begin{equation}
(P_{\ell}^m)'(\pm 1)
<\infty,\quad m\neq 1,\label{Legendre_DERextremes}
\end{equation}
\begin{equation}
(P_\ell^1)'(x)
=\frac{xP_\ell'(x)}{(1-x^2)^{1/2}}+o((1-x^2)^{1/2}),\qquad x\rightarrow\pm1,\label{def:derivativeP1}
\end{equation}
where $P_{\ell}$ is the Legendre polynomial of degree $\ell$.

Since the associated Legendre polynomials are the canonical solutions of the general Legendre equation
\begin{equation}\label{eq:legendre}
\Big((1-x^{2})(P_{\ell }^{m})'\Big)'+\left[\ell (\ell +1)-{\frac {m^{2}}{1-x^{2}}}\right]P_{\ell }^{m}=0,
\end{equation}
multiplying \eqref{eq:legendre} by $P_{\ell'}^m$, we obtain
\begin{equation*}
\Big((1-x^{2})(P_{\ell }^{m})'P_{\ell' }^{m}\Big)'+\left[\ell (\ell +1)-{\frac {m^{2}}{1-x^{2}}}\right]P_{\ell }^{m}P_{\ell' }^{m}=(1-x^2)(P_{\ell }^{m})'(P_{\ell' }^{m})'.
%\label{eq:legendre3}
\end{equation*}
Estimate \eqref{ort_grad-norma1} follows from here using this identity in \eqref{gradP_gradP} and the orthogonality property of the associated Legendre functions, that is,
\begin{equation}
\int_{-1}^1P_{\ell}^mP_{\ell'}^{m}\,dx=\frac{\Omega_{\ell m}}{2\pi}\delta_{\ell\ell'},\quad 0\leq m\leq\ell.\label{ort:Pl}
\end{equation}
Notice that in case $|m|=1$ we have used the fact that
\begin{equation}
P_{\ell}^m(1)=P_{\ell}^m(-1)=0,\quad m>0.\label{Legendre_extremes}
\end{equation}

\noindent
\underline{Proof of \eqref{ort_grad-norma2}}. Arguing as before, we compute
\begin{align}
\nabla_S{\bf B}_\ell^m(\theta,\varphi)=
&\frac{\Omega^{-\frac{1}{2}}_{\ell m}e^{im\varphi}}{[\ell(\ell+1)]^{\frac{1}{2}}}
\left[\sin\theta (P_\ell^m)'(\cos\theta){\hat {\bm \theta}}\otimes{\bf \hat r}
\right.\nonumber\\
&+\left(\sin^2\theta(P_\ell^m)''(\cos\theta)-\cos\theta (P_\ell^m)'(\cos\theta) \right){\hat {\bm \theta}}\otimes{\hat {\bm\theta}}\nonumber\\
&-im\left((P_\ell^m)'(\cos\theta)+\frac{\cos\theta}{\sin^2\theta}P_\ell^m(\cos\theta) \right){\hat {\bm \theta}}\otimes{\hat{\bm\varphi}}
\nonumber\\
&-\frac{im}{\sin\theta}P_\ell^m(\cos\theta){\hat{\bm\varphi}}\otimes{\bf \hat r}\nonumber\\
&-im\left((P_\ell^m)'(\cos\theta)+\frac{\cos\theta}{\sin^2\theta}P_\ell^m(\cos\theta) \right){\hat{\bm\varphi}}\otimes{\hat {\bm\theta}}\nonumber\\
&-\left.\left(\frac{m^2}{\sin^2\theta}P_\ell^m(\cos\theta)+\cos\theta(P_\ell^m)'(\cos\theta)\right){\hat{\bm\varphi}}\otimes{\hat{\bm\varphi}}
\right],
\label{def:gradB}
\end{align}
and
\begin{align}\label{def:gradC}
\nabla_S{\bf C}_\ell^m(\theta,\varphi)=
&\frac{\Omega^{-\frac{1}{2}}_{\ell m}e^{im\varphi}}{[\ell(\ell+1)]^{\frac{1}{2}}}
\left[-\frac{im}{\sin\theta} P_\ell^m(\cos\theta){\hat {\bm \theta}}\otimes{\bf \hat r}
\right.\nonumber
\\
&-im\left((P_\ell^m)'(\cos\theta)+\frac{\cos\theta}{\sin^2\theta} P_\ell^m(\cos\theta) \right){\hat {\bm \theta}}\otimes{\hat {\bm\theta}}\nonumber\\
&+\left(\cos\theta(P_\ell^m)'(\cos\theta)-\sin^2\theta (P_\ell^m)''(\cos\theta) \right){\hat {\bm \theta}}\otimes{\hat{\bm\varphi}}\nonumber\\
&-\sin\theta (P_\ell^m)'(\cos\theta){\hat{\bm\varphi}}\otimes{\bf \hat r}
\nonumber
\\
&-\left(\frac{m^2}{\sin^2\theta}P_\ell^m(\cos\theta)+\cos\theta (P_\ell^m)'(\cos\theta) \right){\hat{\bm\varphi}}\otimes{\hat{\bm \theta}}\nonumber\\
&+\left.im\left((P_\ell^m)'(\cos\theta)+\frac{\cos\theta}{\sin^2\theta}P_\ell^m(\cos\theta)\right){\hat{\bm\varphi}}\otimes{\hat{\bm\varphi}}\right],
\end{align}
and thus
\begin{align}
T
&:=
\langle {\nabla_S{\bf B}_\ell^m,\nabla_S{\bf B}_{\ell'}^{m'}} \rangle_{L^2(\mathbb{S}^2)}
=
\langle {\nabla_S{\bf C}_\ell^m,\nabla_S{\bf C}_{\ell'}^{m'}} \rangle_{L^2(\mathbb{S}^2)}=
\nonumber
\\
&=\Omega\,\delta_{mm'}[\ell(\ell+1)]^{-\frac{1}{2}}[\ell'(\ell'+1)]^{-\frac{1}{2}}(I_1+I_2+I_3+I_4+I_5),
\label{gradB_gradB}
\end{align}
with $\Omega$ given in \eqref{Omega},
\begin{eqnarray}
%\label{I1_lemaSphericalB}
I_1&=&
\int_{-1}^1(1-x^2)^2\,(P_{\ell}^m)''(x)(P_{\ell'}^m)''(x)\,dx,\nonumber
\\
%\label{I2_lemaSphericalB}
I_2&=&
\int_{-1}^1(1-x^2)x\,\Big((P_{\ell}^m)'(P_{\ell'}^m)'\Big)'(x)\,dx,\nonumber
\\
%\label{I3_lemaSphericalB}
I_3&=&
\int_{-1}^1\frac{3m^2x}{(1-x^2)}\,\Big(P_{\ell}^mP_{\ell'}^m\Big)'(x) \,dx,\nonumber
\\
\label{I4_lemaSphericalB}
I_4&=& \int_{-1}^1(1+x^2+2m^2)(P_{\ell}^m)'(x)(P_{\ell'}^m)'(x)\,dx,
\\
%\label{I5_lemaSphericalB}
I_5&=& \int_{-1}^1\frac{(m^2+m^2x^2+m^4)}{(1-x^2)^2}P_{\ell}^m(x)P_{\ell'}^m(x)\,dx.\nonumber
\end{eqnarray}
We notice that while the previous integrals are definite integrals if $m\neq 1,$ they improper integrals if $m= 1.$ This is due to \eqref{Legendre}, \eqref{Legendre_DERextremes}, \eqref{def:derivativeP1}, 
\begin{equation}
(P_{\ell}^m)''(\pm 1)
<\infty,\quad m\neq 1,\label{Legendre_DER2extremes}
\end{equation}
\begin{equation}
(P_\ell^1)''(x)
=\frac{x^2P_\ell'(x)}{(1-x^2)^{3/2}}+o((1-x^2)^{-1/2}.\label{def:derivative2P1}
\end{equation}
First of all, we will prove the result for the case $m\neq 1$.

Integrating by parts, we have that
\begin{align}
I_2
&=\int_{-1}^1(1-3x^2)(P_{\ell}^m)'(x)(P_{\ell'}^m)'(x)\,dx,
\label{partes_I2}
\\
I_3+I_5
&=\int_{-1}^1\frac{2m^2x}{(1-x^2)}\,\Big(P_{\ell}^mP_{\ell'}^m\Big)'(x) \,dx
+
\int_{-1}^1\frac{m^4}{(1-x^2)^2}P_{\ell}^m(x)P_{\ell'}^m(x)\,dx.
\label{partes_I3}
\end{align}
Notice that the border terms appearing in the integration by parts of the identity \eqref{partes_I3} are equal zero due to \eqref{Legendre_extremes} (in the case $m=0$, $I_3+I_5=0$).

Integrating again by parts and using equation \eqref{eq:legendre} but rewritten as
\begin{equation}\label{eq:legendre2}
(1-x^{2})(P_{\ell }^{m})''-2x(P_{\ell }^{m})'+\left[\ell (\ell +1)-{\frac {m^{2}}{1-x^{2}}}\right]P_{\ell }^{m}=0,
\end{equation}
we obtain that 
\begin{align*}
I_1=
&
\int_{-1}^1\left[\left(\ell(\ell+1)-2\right)(1-x^2)-m^2\right](P_{\ell}^m)'(x)(P_{\ell'}^m)'(x)\,dx
\nonumber
\\
&-\int_{-1}^1\frac{m^2x}{(1-x^2)}\,\Big(P_{\ell}^mP_{\ell'}^m\Big)'(x) \,dx.
%\label{partes_I1}
\end{align*}
By symmetry we can change the role of $\ell$ and $\ell'$ to get
\begin{align}
I_1=
&
\int_{-1}^1\left[\left(\frac{\ell(\ell+1)+\ell'(\ell'+1)}{2}-2\right)(1-x^2)-m^2\right](P_{\ell}^m)'(x)(P_{\ell'}^m)'(x)\,dx
\nonumber
\\
&-\int_{-1}^1\frac{m^2x}{(1-x^2)}\,\Big(P_{\ell}^mP_{\ell'}^m\Big)'(x) \,dx.
\label{partes_I1}
\end{align}
Using \eqref{I4_lemaSphericalB}, \eqref{partes_I2},\eqref{partes_I3} and \eqref{partes_I1} in \eqref{gradB_gradB} we have that
\begin{align}
I&:=I_1+I_2+I_3+I_4+I_5
\nonumber
\\
&\ =
\int_{-1}^1\left(\frac{\ell(\ell+1)+\ell'(\ell'+1)}{2}(1-x^2)+m^2\right)(P_{\ell}^m)'(x)(P_{\ell'}^m)'(x)\,dx
\nonumber
\\
&
\quad 
+
\int_{-1}^1\frac{m^2x}{(1-x^2)}\,\Big(P_{\ell}^mP_{\ell'}^m\Big)'(x) \,dx
+
\int_{-1}^1\frac{m^4}{(1-x^2)^2}P_{\ell}^m(x)P_{\ell'}^m(x)\,dx.
\label{T}
\end{align}
Integrating by parts, using \eqref{eq:legendre2} and the symmetry, we obtain that
\begin{align}
\int_{-1}^1m^2(P_{\ell}^m)'(x)(P_{\ell'}^m)'(x)\,dx=
&
\int_{-1}^1\frac{m^2[\ell(\ell+1)+\ell'(\ell'+1)]}{2(1-x^2)}P_{\ell}^m(x)P_{\ell'}^m(x)\,dx
\nonumber
\\
&-\int_{-1}^1\frac{m^2x}{(1-x^2)}\,\Big(P_{\ell}^mP_{\ell'}^m\Big)'(x) \,dx
\nonumber
\\
&-\int_{-1}^1\frac{m^4}{(1-x^2)^2}P_{\ell}^m(x)P_{\ell'}^m(x)\,dx.
\label{horror}
\end{align}
As before, the border terms appearing in the integration by parts of the identity \eqref{horror} equal zero due to \eqref{Legendre_extremes} (in the case $m=0$, the integral on the left hand side of \eqref{horror} is zero).

On the other hand, using \eqref{eq:legendre}, the symmetry and \eqref{ort:Pl}, we get
\begin{align}
\int_{-1}^1(1-x^2)(P_{\ell}^m)'(x)(P_{\ell'}^m)'(x)\,dx=
&
\frac{[\ell(\ell+1)+\ell'(\ell'+1)]\Omega_{\ell m}}{4\pi}
\delta_{\ell\ell'}
\nonumber
\\
&-\int_{-1}^1\frac{m^2}{1-x^2}P_{\ell}^m(x)P_{\ell'}^m(x)\,dx.
\label{horror2}
\end{align}
Estimate \eqref{ort_grad-norma2} follows from \eqref{gradB_gradB}, \eqref{T}, \eqref{horror} and \eqref{horror2}.

In case $m= 1$, we are dealing with improper integrals, but the strategy used for the case $m\neq 1$ is still valid if we replace the limits of integration $-1$ and $1$ by $a$ and $b$, and then make $a$ tends to $-1$ and $b$ tends to $1$. In this procedure, we get border terms depending on $a$ and $b$ which cancel before taking the limits, and therefore, we get the same result as before. We leave the details for the reader.

\medskip
\noindent
\underline{Proof of \eqref{ort_grad-norma3}}. From \eqref{def:gradP} and \eqref{def:gradC}, arguing as before and using \eqref{Legendre_extremes}, we obtain that
\begin{align*}
\langle {\nabla_S{\bf P}_\ell^m,\nabla_S{\bf C}_{\ell'}^{m'}} \rangle_{L^2(\mathbb{S}^2)}=
-im\Omega[\ell'(\ell'+1)]^{-\frac{1}{2}}\delta_{mm'}\int_{-1}^{1}
\Big(P_{\ell}^mP_{\ell'}^m\Big)'(x)\,dx=0,
\end{align*}
where $\Omega$ is given in \eqref{Omega}.

For the second inner product in \eqref{ort_grad-norma3}, we use \eqref{def:gradB} and \eqref{def:gradC} to write
\begin{equation}
\langle {\nabla_S{\bf B}_\ell^m,\nabla_S{\bf C}_{\ell'}^{m'}} \rangle_{L^2(\mathbb{S}^2)}=
im\Omega\,\delta_{mm'}[\ell(\ell+1)]^{-\frac{1}{2}}[\ell'(\ell'+1)]^{-\frac{1}{2}}(S_1+S_2+S_3+S_4),
\label{gradB_gradC}
\end{equation}
with $\Omega$ given in \eqref{Omega},
\begin{eqnarray}
%\label{S1}
S_1&=&
\int_{-1}^1\left(1+\frac{m^2}{1-x^2}\right)\Big(P_{\ell}^mP_{\ell'}^m\Big)'(x) \,dx,\nonumber
\\
%\label{S2}
S_2&=&
\int_{-1}^1(1-x^2)\,\Big((P_{\ell}^m)'(P_{\ell'}^m)'\Big)'(x)\,dx,\nonumber
\\
%\label{S3}
S_3&=&
\int_{-1}^1x\,\Big[(P_{\ell}^m)''P_{\ell'}^m+P_{\ell}^m(P_{\ell'}^m)''\Big](x)\,dx,\nonumber
\\
\label{S4}
S_4&=&
\int_{-1}^1\frac{2m^2}{(1-x^2)^2}P_{\ell}^m(x)P_{\ell'}^m(x)\,dx.
\end{eqnarray}

We notice that the previous integrals are definite integrals if $m\neq 1,$ but are improper integrals if $m= 1.$ This is due to \eqref{Legendre}, \eqref{Legendre_DERextremes}, \eqref{def:derivativeP1}, \eqref{Legendre_DERextremes}, \eqref{Legendre_DER2extremes} and \eqref{def:derivative2P1}. 
First of all, we will prove the result for $m\neq 1$ and $0$. 
The case $m=0$ is trivial.

Using \eqref{Legendre_extremes}, \eqref{Legendre}, \eqref{Legendre_DERextremes}, \eqref{Legendre_extremes} and integrating by parts we obtain that
\begin{align}
\label{partes_S1}
S_1&=-\int_{-1}^1\frac{2m^2}{(1-x^2)^2}P_{\ell}^m(x)P_{\ell'}^m(x)\,dx,
\\
\label{partes_S2_S3}
S_2&=-\int_{-1}^12x\,(P_{\ell}^m)'(x)(P_{\ell'}^m)'(x)\,dx=-S_3.
\end{align}
The result follows form \eqref{gradB_gradC}, \eqref{S4}, \eqref{partes_S1} and \eqref{partes_S2_S3}.

The case $m= 1$, as it happens in the proof of \eqref{ort_grad-norma2}, some border terms appear when integration by parts is used, but they tend to zero when taking the limits. Actually, using \eqref{Legendre} and \eqref{def:derivativeP1}, the reader can check, after a few computations, that the border terms obtained are the following:
$$
(1-b^2)P_\ell'(b)-(1-a^2)P_\ell'(a)+o((1-b^2))+o((1-a^2)).
$$  

\medskip
\noindent
\underline{Proof of \eqref{ort_grad-norma4}}. From \eqref{def:gradP} and \eqref{def:gradB}, we obtain that
\begin{equation*}
%\label{gradP_gradB}
\langle {\nabla_S{\bf P}_\ell^m,\nabla_S{\bf B}_{\ell'}^{m'}} \rangle_{L^2(\mathbb{S}^2)}
=
\Omega\,\delta_{mm'}[\ell'(\ell'+1)]^{-\frac{1}{2}}(F_1+F_2+F_3-S_4),
\end{equation*}
with $\Omega$ and  $S_4$ given in \eqref{Omega} and \eqref{S4} respectively,
\begin{eqnarray*}
	%\label{F1}
	F_1&=&
	-\int_{-1}^1(1-x^2)\,(P_{\ell}^m)'(x)(P_{\ell'}^m)'(x)\,dx,
	\\
	%\label{F2}
	F_2&=&
	\int_{-1}^1(1-x^2)\,P_{\ell}^m(x)(P_{\ell'}^m)''(x)\,dx,
	\\
	%\label{F3}
	F_3&=&
	\int_{-1}^12x\,P_{\ell}^m(x)(P_{\ell'}^m)'(x)\,dx.
\end{eqnarray*}
The result follows integrating by parts $F_2$, and using \eqref{Legendre_extremes}, \eqref{Legendre_DERextremes}, \eqref{def:derivativeP1} and \eqref{ort:Pl}.
\hfill$\Box$

\medskip

We finish the section with the proof of \eqref{a} in Lemma \ref{integralesBessel1}.

\medskip

\emph{Proof of Lemma \ref{integralesBessel1}.} To prove \eqref{a}, we only detail the case $0<a<b$, since the case $0<b<a$ is similar. First of all, we use the relationship between the spherical Bessel functions and the Bessel functions given in  \eqref{Jj} to write
\begin{equation}
\label{jl}
I=\int_0^\infty j_\ell(ar) j_{\ell}(br)\,r^2\,\frac{dr}{\langle r\rangle^3}
\sim
\int_0^\infty J_{\ell+\frac{1}{2}}(ar)J_{\ell+\frac{1}{2}}(br) \,r\,\frac{dr}{\langle r\rangle^3}.
\end{equation}
We split the last integral in \eqref{jl} into two, to write
\begin{equation}
\label{divido}I\sim I_{1}+I_{2},
\end{equation}
with
\begin{equation*}
%\label{lemma2I1I2}
I_1=
\int_{0}^{1}J_{\ell+\frac{1}{2}}(ar) J_{\ell+\frac{1}{2}}(br)  \,r\frac{dr}{\langle r\rangle^3},
\qquad
I_2=
\int_{1}^{\infty}J_{\ell+\frac{1}{2}}(ar) J_{\ell+\frac{1}{2}}(br)  \,r\frac{dr}{\langle r\rangle^3}.
\end{equation*}

Using \eqref{CotaBessel}, we have that
\begin{align}
\label{I1_ab}
I_1\sim \int_{0}^{1}J_{\ell+\frac{1}{2}}(ar) J_{\ell+\frac{1}{2}}(br) r \,dr\sim\frac{1}{\left(\ell+\frac{1}{2}\right)^{2(\ell+\frac{3}{2})}}\frac{e^{2\ell}}{2^{2\ell}}a^{\ell+\frac{1}{2}}\,b^{\ell+\frac{1}{2}}.
\end{align}
On the other hand, 
\begin{equation}
\label{I2ab}
I_2\sim\int_{1}^{\infty}J_{\ell+\frac{1}{2}}(ar) J_{\ell+\frac{1}{2}}(br)  \,\frac{dr}{r^2}=I_3-I_4,
\end{equation}
with
\begin{equation*}
%\label{lemma2I3I4}
I_3=
\int_{0}^{\infty}J_{\ell+\frac{1}{2}}(ar) J_{\ell+\frac{1}{2}}(br)  \,\frac{dr}{r^2},
\qquad
I_4=
\int_{0}^{1}J_{\ell+\frac{1}{2}}(ar) J_{\ell+\frac{1}{2}}(br)  \,\frac{dr}{r^2}.
\end{equation*}

Arguing as we did to get \eqref{I1_ab}, we obtain that
\begin{equation}
\label{I4_ab}
|I_4|\sim  \int_{0}^{1}J_{\ell+\frac{1}{2}}(ar) J_{\ell+\frac{1}{2}}(br) \,dr\sim I_1
\end{equation}
In order to estimate $I_3$, we observe that it is a Weber-Schafheitlin type integral. In \cite{abramowitz} (see 11.4.34 on page 487) we can find the following identity for this type of integrals: whenever $0<a<b,$ $\Re \lambda>-1$ and $\Re(\mu+\nu-\lambda+1)>0,$
\begin{equation*}
\label{Weber}
\int_0^\infty J_{\mu}(ar) J_{\nu}(br)\,\frac{dr}{r^{\lambda}}
=
\frac{_2F_1\left(\frac{\mu+\nu-\lambda+1}{2},\frac{\mu-\nu-\lambda+1}{2};\mu+1;\frac{a^2}{b^2}\right)a^{\mu}\Gamma\left(\frac{\mu+\nu-\lambda+1}{2}\right)}{2^{\lambda}b^{\mu-\lambda+1}\Gamma(\mu+1)\Gamma\left(\frac{\nu-\mu+\lambda+1}{2}\right)},
\end{equation*}
where $_2F_1$ is a Gauss hypergeometric series defined by (see 15.1.1 on page 556 of \cite{abramowitz})
\begin{equation*}
\label{hipergeometrica}
_2F_1(a,b;c;z)=\frac{\Gamma(c)}{\Gamma(a)\Gamma(b)}
\sum_{n=0}^{\infty}\frac{\Gamma(a+n)\Gamma(b+n)z^n}{\Gamma(c+n)n!},
\end{equation*}
which is absolutely convergent for $|z|\leq 1$ and $\Re (c-a-b)>0.$

From here, we obtain for any $\ell>0,$
\begin{equation*}
I_3
\sim
\frac{a^\ell}{b^\ell\Gamma\left(-\frac{1}{2}\right)}\, 
\sum_{k=0}^{\infty}\frac{\Gamma(\ell+k)\Gamma\left(-\frac{1}{2}+k\right)a^{2k}}
{\Gamma\left(\ell+\frac{3}{2}+k\right)k!b^{2k}}.
\end{equation*}
Using now the well known recurrence formula for the Gamma function
$$
\Gamma(z+1)=z\Gamma(z),
$$
and the duplication formula  (see 6.1.18 on page 256 of \cite{abramowitz})
$$
\Gamma(2z)=(2\pi)^{-\frac{1}{2}}2^{2z-\frac{1}{2}}\Gamma(z)\Gamma\left(z+\frac{1}{2}\right),
$$
we get
\begin{equation*}
I_3
\sim
\left(\frac{a}{b}\right)^\ell 2^{2\ell}\, 
\sum_{k=0}^{\infty}\frac{\left(\Gamma(\ell+k)\right)^22^{2k}\Gamma\left(-\frac{1}{2}+k\right)a^{2k}}
{(\ell+k)\Gamma\left(2\ell+2k\right)\Gamma\left(-\frac{1}{2}\right)k!b^{2k}}.
\end{equation*}
Using the identity (see 6.1.22 on page 256 of \cite{abramowitz})
$$
\frac{\Gamma(z+k)}{\Gamma(z)}=z(z+1)\ldots(z+k-1),
$$
we have that
\begin{equation*}
I_3
\sim
\left(\frac{a}{b}\right)^\ell 2^{2\ell}\, 
\sum_{k=0}^{\infty}\frac{\left(\Gamma(\ell+k)\right)^22^{2k}a^{2k}}
{(\ell+k)\Gamma\left(2\ell+2k\right)b^{2k}}.
\end{equation*}
Since $\Gamma(n)=(n-1)!$,  using the Stirling formula given by
\begin{equation}
\label{stirling}
n!\sim n^{n+\frac{1}{2}}e^{-n},\qquad n\in\mathbb{N},\,n\rightarrow\infty,
\end{equation}
we obtain that
\begin{equation*}
I_3
\sim
\left(\frac{a}{b}\right)^\ell \, 
\sum_{k=0}^{\infty}\frac{1}{(\ell+k)^{\frac{3}{2}}}\,\left(\frac{a}{b}\right)^{2k},
\end{equation*}
and therefore, since $0<a<b,$ we have that
\begin{equation}
\label{I3ab}
I_3
\sim
\left(\frac{a}{b}\right)^\ell \, 
\ell^{-\frac{3}{2}}.
\end{equation}
The result follows from \eqref{I1_ab}-\eqref{I3ab},  because we have that
\begin{equation*}
I\sim\left(\frac{a}{b}\right)^\ell \, 
\ell^{-\frac{3}{2}}\left(1+\left(\frac{b\,e}{a\,2\,\ell}\right)^{2\ell}\frac{1}{\ell^{\frac{3}{2}}}\right)
\end{equation*}
and the second term inside the parentheses is $0$ for $\ell$ large enough.
\hfill $\square$
\section{Reproducing kernel of $\mathcal{W}(\mathbb{R}^3)$ and $\mathcal{W}(\mathbb{R}^2)$} 
In this section we will construct the reproducing kernel for $\mathcal{W}(\mathbb{R}^3)$ as a subspace of the Hilbert space $\mathcal{H}(\mathbb{R}^3)$. 
The construction of the reproducing kernel for $\mathcal{W}(\mathbb{R}^2)$ is similar to the previous one but easier, so we just sketch it.

We start introducing the definition of a complex-vector-valued reproducing kernel Hilbert space. For a complete study about this topic, we refer the reader to \cite{vvRKHS}, \cite{vvRKHS2} and \cite{vvRKHS3}. 

\begin{definition}
	\label{vvRKHS}
	Let $H$ be a Hilbert space of complex-vector-valued functions 
	$\mathbf{u}\,:\,X\longrightarrow \mathbb{C}^d,$ $d\in\mathbb{N},$ 
	with inner product $\langle\cdot,\cdot\rangle_H.$ We say that $H$ is a reproducing kernel Hilbert space if there exists a complex-valued second-order tensor $\Gamma(x,y),$ with $(x,y)\in X\times X$ satisfying that
	\begin{itemize}
		\item [(i)]
		For all $x\in X$ and $\bm{z}\in \mathbb{C}^d,$ 
		\begin{equation*}
		\label{condicion1General}
		\Gamma(x,y)z\in H.
		\end{equation*}
		\item [(ii)]
		For all $\mathbf{u}\in H,$ $x\in X$ and $\bm{z}\in \mathbb{C}^d,$
		\begin{equation*}
		\label{condicion2general}
		\mathbf{u}(x)\cdot\overline{\bm{z}}
		=
		\langle \mathbf{u},\Gamma(x,\cdot)\bm{z}\rangle_H.
		\end{equation*}
	\end{itemize} 
	The tensor $\Gamma(x,y)$ is called the reproducing kernel of $H.$
\end{definition}
\begin{remark}
	\label{generaH}
	Notice that $H$ coincides with the closure of the span of the set
	$$\{\Gamma(x,\cdot)\bm{z}\,:\, x\in X,\, \bm{z}\in\mathbb{C}^d\}.$$
\end{remark}
\textit{Proof of Theorem \ref{NucleoReporductor3d}}.
From the identities \eqref{def:NNvectors} and \eqref{DesarrolloLame}, and Proposition \ref{ortoLMN},$\{\mathbfcal{L}_\ell^m$, $\mathbfcal{M}_\ell^m$,$\mathbfcal{N}_\ell^m\}_{\ell,m}$ is an almost orthonormal basis of $\mathcal{W}(\mathbb{R}^3)$. We can get an orthonormal basis by replacing $\mathbfcal{N}_\ell^m$ by $\widetilde{\mathbfcal{N}}_\ell^m$ with
\begin{equation}
\label{Ntilde}
\widetilde{\mathbfcal{N}}_\ell^m=\mathbfcal{N}_\ell^m-
\langle {\mathbfcal{N}_\ell^m,\mathbfcal{L}_{\ell}^{m}} \rangle_{\mathcal{H}}\,\mathbfcal{L}_{\ell}^{m}.
\end{equation}
Notice that from \eqref{LNdif},
\begin{equation}
\label{NormaNtilde}
\|\widetilde{\mathbfcal{N}}_\ell^m\|_{\mathcal{H}}^2=1-|\langle {\mathbfcal{N}_\ell^m,\mathbfcal{L}_{\ell}^{m}} \rangle_{\mathcal{H}}|^2\neq 0.
\end{equation}
Then, since $\{\mathbfcal{L}_\ell^m, \mathbfcal{M}_\ell^m,\widetilde{\mathbfcal{N}}_\ell^m/\|\widetilde{\mathbfcal{N}}_\ell^m\|_{\mathcal{H}}\}_{\ell,m}$ is an orthonormal basis of $\mathcal{W}(\mathbb{R}^3)$, the orthogonal projection $\mathcal{P}$
of $\mathcal{H}(\mathbb{R}^3)$ onto $\mathcal{W}(\mathbb{R}^3)$ is given for any $\mathbf{u}\in \mathcal{H}(\mathbb{R}^3)$ by
\begin{equation*}
\label{ProyeccionOrotogonal}
\displaystyle 
\mathcal{P}\mathbf{u}
=\sum_{\ell=0}^\infty\sum_{m=-\ell}^\ell 
\left(
\langle \mathbf{u},\mathbfcal{L}_{\ell}^{m} \rangle_{\mathcal{H}}\,\mathbfcal{L}_{\ell}^{m}
+
\langle \mathbf{u},\mathbfcal{M}_{\ell}^{m} \rangle_{\mathcal{H}}\,\mathbfcal{M}_{\ell}^{m}
+
\frac{\langle \mathbf{u},\widetilde{\mathbfcal{N}}_\ell^m \rangle_{\mathcal{H}}}{\|\widetilde{\mathbfcal{N}}_\ell^m\|_{\mathcal{H}}^2}\,\widetilde{\mathbfcal{N}}_\ell^m
\right),
\end{equation*}
with convergence in $\mathcal{W}(\mathbb{R}^3)$ and also pointwise. 

In particular, if $\mathbf{u}\in \mathcal{W}(\mathbb{R}^3)$, formally, for any $\bm{x}\in\mathbb{R}^3$ and $\bm{z}\in\mathbb{C}^3$ we can write
that
\begin{equation}
\label{condicion2}
\mathbf{u}(\bm{x})\cdot\overline{\bm{z}}
=
\mathcal{P}\mathbf{u}(\bm{x})\cdot\overline{\bm{z}}
=
\langle \mathbf{u},\Gamma(\bm{x},\cdot)\bm{z}\rangle_{\mathcal{H}},
\end{equation}
where
\begin{equation}
\label{Gammaxy}
\Gamma(\bm{x},\bm{y})\bm{z}
:=\sum_{\ell=0}^\infty\sum_{m=-\ell}^\ell 
\Gamma_{\ell}^{m}(\bm{x},\bm{y})\bm{z},
\end{equation}
with
$$
\Gamma_{\ell}^{m}(\bm{x},\bm{y})\bm{z}
=
\overline{\mathbfcal{L}_{\ell}^{m}(\bm{x})}\cdot\bm{z}\,
\mathbfcal{L}_{\ell}^{m}(\bm{y})
+
\overline{\mathbfcal{M}_{\ell}^{m}(\bm{x})}\cdot\bm{z}\,
\mathbfcal{M}_{\ell}^{m}(\bm{y})
+
\frac{\overline{\widetilde{\mathbfcal{N}}_{\ell}^{m}(\bm{x})}\cdot\bm{z}}{\|\widetilde{\mathbfcal{N}}_\ell^m\|_{\mathcal{H}}^2}\,\widetilde{\mathbfcal{N}}_\ell^m(\bm{y}).
$$
For convenience, we will write $\Gamma_{\ell}^{m}(\bm{x},\bm{y})$ as the following second order tensor
\begin{equation*}
\label{Gammalm}
\Gamma_{\ell}^{m}(\bm{x},\bm{y})
:=
\mathbfcal{L}_{\ell}^{m}(\bm{y})\otimes\overline{\mathbfcal{L}_{\ell}^{m}(\bm{x})}
+
\mathbfcal{M}_{\ell}^{m}(\bm{y})\otimes\overline{\mathbfcal{M}_{\ell}^{m}(\bm{x})}
+
\frac{\widetilde{\mathbfcal{N}}_{\ell}^{m}(\bm{y})\otimes\overline{\widetilde{\mathbfcal{N}}_{\ell}^{m}(\bm{x})}}{\|\widetilde{\mathbfcal{N}}_\ell^m\|_{\mathcal{H}}^2}.
\end{equation*}
In order to make \eqref{condicion2} rigorous we have to prove that
the series given in \eqref{Gammaxy} converges absolutely and uniformly on compact subsets of $\mathbb{R}^3\times\mathbb{R}^3.$ 

In order to do that, we use \eqref{Ntilde} and \eqref{NormaNtilde} to write
\begin{eqnarray}
\Gamma_{\ell}^{m}(\bm{x},\bm{y})&=&
\frac{1}{\|\widetilde{\mathbfcal{N}}_\ell^m\|_{\mathcal{H}}^2}
\left(
\mathbfcal{L}_{\ell}^{m}(\bm{y})\otimes\overline{\mathbfcal{L}_{\ell}^{m}(\bm{x})}
+
\mathbfcal{N}_{\ell}^{m}(\bm{y})\otimes\overline{\mathbfcal{N}_{\ell}^{m}(\bm{x})}
\right)
\nonumber
\\
&&
-
\frac{1}{\|\widetilde{\mathbfcal{N}}_\ell^m\|_{\mathcal{H}}^2}
\langle {\mathbfcal{N}_\ell^m,\mathbfcal{L}_{\ell}^{m}} \rangle_{\mathcal{H}}\,
\mathbfcal{L}_{\ell}^{m}(\bm{y})\otimes\overline{\mathbfcal{N}_{\ell}^{m}(\bm{x})}
\nonumber
\\
&&
-
\frac{1}{\|\widetilde{\mathbfcal{N}}_\ell^m\|_{\mathcal{H}}^2}
\overline{\langle {\mathbfcal{N}_\ell^m,\mathbfcal{L}_{\ell}^{m}}} \rangle_{\mathcal{H}}\,
\mathbfcal{N}_{\ell}^{m}(\bm{y})\otimes\overline{\mathbfcal{L}_{\ell}^{m}(\bm{x})}
\nonumber
\\
\label{nucleo}
&&
+
\mathbfcal{M}_{\ell}^{m}(\bm{y})\otimes\overline{\mathbfcal{M}_{\ell}^{m}(\bm{x})}
\end{eqnarray}

From \eqref{NormaNtilde}, \eqref{def:NNvectors}, \eqref{ort-norma-lmn-3} and \eqref{NormasUnitarios}, we have that
\begin{equation}
\label{cotas}
\|\widetilde{\mathbfcal{N}}_\ell^m\|_{\mathcal{H}}=O(1)
\qquad
\text{and}
\qquad
\langle {\mathbfcal{N}_\ell^m,\mathbfcal{L}_{\ell}^{m}} \rangle_{\mathcal{H}}=O(\ell^{-2}).
\end{equation}

From \eqref{nucleo}, using the triangle inequality, the Cauchy-Schwarz inequality and \eqref{cotas}, one can see that it is enough to prove that the series given in \eqref{series}  
converges absolutely and uniformly on compact subsets of $\mathbb{R}^3.$ This was done in Remark \ref{convergenciaSerieLame}.

To finish, we have to prove that for any $\bm{x}\in\mathbb{R}^3$ and $\bm{z}\in\mathbb{C}^3$ the series
\begin{equation*}
\label{Gamma}
\Gamma(\bm{x},\cdot)\bm{z}
=\sum_{\ell=0}^\infty\sum_{m=-\ell}^\ell 
\Gamma_{\ell}^{m}(\bm{x},\cdot)\bm{z}
\end{equation*}
converges in $\mathcal{W}(\mathbb{R}^3)$. 

In order to do this, it is enough to prove that
\begin{equation*}
\sum_{\ell=0}^\infty\sum_{m=-\ell}^\ell 
\|\Gamma_{\ell}^{m}(\bm{x},\cdot)\bm{z}\|_{\mathcal{H}}<\infty.
\end{equation*}

Actually, from the definition of $\Gamma_{\ell}^{m}$ given in \eqref{nucleo}, using the triangle inequality and \eqref{cotas}, one can see that it is enough to prove that the series
\begin{equation*}
\sum_{\ell=0}^\infty\sum_{m=-\ell}^\ell 
|\mathbfcal{L}_{\ell}^{m}(\bm{x})|,
\qquad
\sum_{\ell=0}^\infty\sum_{m=-\ell}^\ell
|\mathbfcal{M}_{\ell}^{m}(\bm{x})|,
\qquad
\sum_{\ell=0}^\infty\sum_{m=-\ell}^\ell
|\mathbfcal{N}_{\ell}^{m}(\bm{x})|
\end{equation*}
converge for any $\bm{x}\in\mathbb{R}^3$ fixed.

Moreover, arguing as in Remark \ref{ConvergenciaMadre}, it is enough to prove that
\begin{equation*}
\sum_{\ell=0}^\infty \ell\, \left|j_l(r)\right|<\infty.
\end{equation*}
This can be done using \eqref{Jj} and \eqref{CotaBessel}.
\hfill $\square$
\begin{remark}
	\label{DefNucleo3d}
	The reproducing kernel of $\mathcal{W}(\mathbb{R}^3)$ as a subspace of the Hilbert space $\mathcal{H}(\mathbb{R}^3)$ is the complex-valued second-order tensor $\Gamma(\bm{x},\bm{y})$ defined by \eqref{Gammaxy} and \eqref{nucleo}.
	
	Notice that for any $\bm{x},\bm{y}\in\mathbb{R}^3,\,\bm{z}\in\mathbb{C}^3,$ we have that
	$$
	\Gamma(\bm{y},\bm{x})=\overline{\Gamma(\bm{x},\bm{y})}\qquad \text{and}\qquad
	\bm{z}\cdot\Gamma(\bm{x},\bm{x})\overline{\bm{z}}\ge 0.
	$$
	Moreover, for any $\bm{z}_1,\bm{z}_2\in\mathbb{C}^3,$ we have that
	\begin{equation*}
	\Gamma(\bm{x},\bm{y})\bm{z}_1\cdot\overline{\bm{z}_2}
	=
	\langle \Gamma(\bm{x},\cdot)\bm{z}_1,\Gamma(\bm{y},\cdot)\bm{z}_2\rangle
	=
	\overline{\langle\Gamma(\bm{y},\cdot)\bm{z}_2, \Gamma(\bm{x},\cdot)\bm{z}_1\rangle}
	=
	\bm{z}_1\cdot\overline{\Gamma(\bm{y},\bm{x})\bm{z}_2}.
	\end{equation*}
\end{remark}
We finish this section sketching the proof of Theorem \ref{NucleoReporductor3d}.

Before doing that, following \cite[p.~5]{2d}, for $n\in \mathbb{Z},$ we introduce the functions
\begin{equation*}
\label{F}
F_{n,k}(\bm{x})= F_{n,k}(r\cos\theta,r\sin\theta):=J_n(rk)\,e^{in\theta},\qquad r>0,\ \theta\in[0,2\pi)
\end{equation*}
and
\begin{equation*}
\label{base2d}
\mathbf{e}_n:=\frac{\nabla F_{n,k_p}}{\|\nabla F_{n,k_p}\|_{\mathcal{H}}},\qquad
\mathbf{f}_n:=\frac{\nabla^\perp F_{n,k_s}}{\|\nabla^\perp F_{n,k_s}\|_{\mathcal{H}}}.
\end{equation*} 
We can write any entire solution of the homogeneous spectral Navier equation in dimension two $\mathbf{u}$ as
\begin{equation}
\label{DesarrolloLame2d}
\displaystyle \mathbf{u}(\bm{x})=\sum_{n\in\mathbb{Z}} 
\left(
a_n\mathbf{e}_n({\bm x})
+
b_n\mathbf{f}_n({\bm x})
\right),
\end{equation}
for certain constants $a_n,b_n.$

As in the three-dimensional case, one can prove that the series given in \eqref{DesarrolloLame2d} converges absolutely and uniformly on compact subsets of $\mathbb{R}^2.$ Moreover, we can differentiate it repeatedly term by term with absolute and uniform convergence in compact subsets of $\mathbb{R}^2.$ 

\hspace{-0.5cm}\textit{Sketch of the proof of Theorem \ref{NucleoReporductor3d}}. 
From Proposition 2.1 of \cite{2d}, 
$\{\mathbf{e}_n, \mathbf{f}_n\}_{n\in\mathbb{Z}}$ is an almost orthonormal basis of $\mathcal{W}(\mathbb{R}^2)$, and therefore, arguing as in the 
proof of Theorem \ref{NucleoReporductor3d}, we can get an orthonormal basis by replacing $\mathbf{f}_n$ by 
\begin{equation*}
\label{ftilde}
\widetilde{\mathbf{f}}_n=\mathbf{f}_n-
\langle \mathbf{f}_n,\mathbf{e}_n \rangle_{\mathcal{H}}\,\mathbf{e}_n.
\end{equation*}
Thus the orthogonal projection $\mathcal{P}$
of $\mathcal{H}(\mathbb{R}^2)$ onto $\mathcal{W}(\mathbb{R}^2)$ is given by
\begin{equation*}
\label{ProyeccionOrotogonal2d}
\displaystyle 
\mathcal{P}\mathbf{u}
=\sum_{n\in\mathbb{Z}}
\left(
\langle \mathbf{u},\mathbf{e}_n \rangle_{\mathcal{H}}\,\mathbf{e}_n
+
\frac{\langle \mathbf{u},\widetilde{\mathbf{f}}_n \rangle_{\mathcal{H}}}{\|\widetilde{\mathbf{f}}_n\|_{\mathcal{H}}^2}\,\widetilde{\mathbf{f}}_n
\right),
\qquad
\mathbf{u}\in\mathcal{H}(\mathbb{R}^2),
\end{equation*}
with convergence in $\mathcal{W}(\mathbb{R}^2)$ and also pointwise.

In particular, if $\mathbf{u}\in \mathcal{W}(\mathbb{R}^2)$, formally, for any $\bm{x}\in\mathbb{R}^2$ and $\bm{z}\in\mathbb{C}^2$ we get 
\begin{equation*}
\label{condicion2_2d}
\mathbf{u}(\bm{x})\cdot\overline{\bm{z}}
=
\mathcal{P}\mathbf{u}(\bm{x})\cdot\overline{\bm{z}}
=
\langle \mathbf{u},\Gamma(\bm{x},\cdot)\bm{z}\rangle_{\mathcal{H}},
\end{equation*}
with
\begin{equation}
\label{Gammaxy2d}
\Gamma(\bm{x},\bm{y})\bm{z}
:=\sum_{n\in\mathbb{Z}}
\Gamma_n(\bm{x},\bm{y})\bm{z},
\end{equation}
and
\begin{equation*}
\label{Gamman}
\Gamma_n(\bm{x},\bm{y})
:=
\mathbf{e}_n(\bm{y})\otimes\overline{\mathbf{e}_n(\bm{x})}
+
\frac{\widetilde{\mathbf{f}}_n(\bm{y})\otimes\overline{\widetilde{\mathbf{f}}_n(\bm{x})}}{\|\widetilde{\mathbf{f}}_n\|_{\mathcal{H}}^2}.
\end{equation*}
To prove that the series given in \eqref{Gammaxy2d} converges absolutely and uniformly on compact subsets of $\mathbb{R}^2\times\mathbb{R}^2$ we write
\begin{eqnarray}
\Gamma_n(\bm{x},\bm{y})&=&
\frac{1}{\|\widetilde{\mathbf{f}}_n\|_{\mathcal{H}}^2}
\left(
\mathbf{e}_n(\bm{y})\otimes\overline{\mathbf{e}_n(\bm{x})}
+
\mathbf{f}_n(\bm{y})\otimes\overline{\mathbf{f}_n(\bm{x})}
\right)
\nonumber
\\
&&
-
\frac{1}{\|\widetilde{\mathbf{f}}_n\|_{\mathcal{H}}^2}
\langle {\mathbf{f}_n,\mathbf{e}_n} \rangle_{\mathcal{H}}\,
\mathbf{e}_n(\bm{y})\otimes\overline{\mathbf{f}_n(\bm{x})}
\nonumber
\\
&&
-
\frac{1}{\|\widetilde{\mathbf{f}}_n\|_{\mathcal{H}}^2}
\overline{\langle {\mathbf{f}_n,\mathbf{e}_n} \rangle_{\mathcal{H}}}\,
\mathbf{f}_n(\bm{y})\otimes\overline{\mathbf{e}_n(\bm{x})}.
\label{nucleo2d}
\end{eqnarray}
Following the proof of Proposition 2.1 of \cite{2d} one can get that
\begin{equation}
\label{cota_ef_2d}
\|\nabla F_{n,k_p}\|_{\mathcal{H}}=\|\nabla^\perp F_{n,k_s}\|_{\mathcal{H}}=O(1).
\end{equation}
Besides, from the identity (26) of Proposition 2.1 of \cite{2d}, we have that
\begin{equation}
\label{cotas2d}
\langle {\mathbf{f}_n,\mathbf{e}_n} \rangle_{\mathcal{H}}=o(1)
\qquad
\text{and}
\qquad
\|\widetilde{\mathbf{f}}_n\|_{\mathcal{H}}=O(1).
\end{equation} 
From \eqref{nucleo2d}, \eqref{cota_ef_2d} and \eqref{cotas2d} one can see that, to prove the convergence result for the series in \eqref{Gammaxy2d} it is enough to prove that the series 
\begin{equation*}
\label{ConvergeBessel2d}
\sum_{n\in\mathbb{Z}}\left(J_n(r)\right)^2
\end{equation*}
converges absolutely and uniformly on compact subsets of $(0,\infty).$ This can be proved using \eqref{CotaBessel}.

Finally, for any $\bm{x}\in\mathbb{R}^2$ and $\bm{z}\in\mathbb{C}^2,$ to prove the convergence of the series
\begin{equation*}
\label{Gamma2d}
\Gamma(\bm{x},\cdot)\bm{z}
=\sum_{n\in\mathbb{Z}} 
\Gamma_n(\bm{x},\cdot)\bm{z}
\end{equation*}
in $\mathcal{W}(\mathbb{R}^2)$ it is enough to prove the convergence of the series
\begin{equation*}
\sum_{n\in\mathbb{Z}}\left|J_n(r)\right|.
\end{equation*}
Again, this is can be proved using \eqref{CotaBessel}.
\hfill $\square$
%%%%%%%%%%%%%%%
%%%%%%%%%%%%%%%%%

%\begin{acknowledgements}
%\end{acknowledgements}



\end{document}